\documentclass[a4paper,reqno]{amsart}
\usepackage[utf8]{inputenc}
\usepackage[T1]{fontenc}
\usepackage[british]{babel}
\usepackage{amsmath,amssymb,amsthm}
\usepackage{xcolor}
\usepackage[all]{xy}
\usepackage{mathtools}
\usepackage{extarrows}
\usepackage{stmaryrd}
\usepackage{mathbbol}
\usepackage{mathabx}
\usepackage{hyperref}

\newdir{>}{{}*:(1,-.2)@^{>}*:(1,+.2)@_{>}}
\newdir{<}{{}*:(1,+.2)@^{<}*:(1,-.2)@_{<}}

\newbox\pullbackbox
\setbox\pullbackbox=\hbox{\xy 0;<1mm,0mm>: \POS(4,0)\ar@{-} (0,0)
 \ar@{-} (4,4)
 \endxy}

\def\pullback{
 \ar@{-}[]+R+<6pt,-1pt>;[]+RD+<6pt,-6pt>%
 \ar@{-}[]+D+<1pt,-6pt>;[]+RD+<6pt,-6pt>}

\def\dottedpullback{%
 \ar@{.}[]+R+<6pt,-1pt>;[]+RD+<6pt,-6pt>%
 \ar@{.}[]+D+<1pt,-6pt>;[]+RD+<6pt,-6pt>}
 
\def\pushout{%
 \ar@{-}[]+L+<-6pt,1pt>;[]+LU+<-6pt,6pt>%
 \ar@{-}[]+U+<-1pt,6pt>;[]+LU+<-6pt,6pt>}

\newenvironment{tfae}
{
\begin{enumerate}}
{\end{enumerate}}

\numberwithin{figure}{section}

\newcommand{\noproof}{\hfill\qed}
\renewcommand{\binom}{\genfrac{\lgroup}{\rgroup}{0pt}{1}}
\newcommand{\cosmash}{\diamond}
\newcommand{\priem}[1]{#1'}

\newcommand{\comp}{\raisebox{0.2mm}{\ensuremath{\scriptstyle{\circ}}}}

\newcommand{\Hom}{\mathit{Hom}}
\newcommand{\KP}{\mathit{Eq}}
\newcommand{\cl}{\mathit{cl}}
\newcommand{\Grpd}{\mathbf{Grpd}}
\newcommand{\Ab}{\mathbf{Ab}}
\newcommand{\XMod}{\mathbf{XMod}}
\newcommand{\Pt}{\mathbf{Pt}}
\newcommand{\Act}{\mathbf{Act}}
\newcommand{\RG}{\mathbf{RG}}
\newcommand{\TrivAct}{\mathbf{TrivAct}}
\newcommand{\TXMod}{\mathbf{AAXMod}}
\newcommand{\A}{\mathbb{A}}
\newcommand{\C}{\mathbb{C}}
\newcommand{\X}{\mathbb{X}}
\newcommand{\Ext}{\mathbf{Ext}}
\newcommand{\Sub}{\mathbf{Sub}}
\newcommand{\Triv}{\mathbf{Triv}}
\newcommand{\Centr}{\mathbf{Centr}}

\newcommand{\SH}{{\rm (SH)}}

\theoremstyle{plain} 
\newtheorem{theorem}[subsection]{Theorem}
\newtheorem{corollary}[subsection]{Corollary} 
\newtheorem{lemma}[subsection]{Lemma} 
\newtheorem{proposition}[subsection]{Proposition}

\theoremstyle{definition} 
\newtheorem{definition}[subsection]{Definition} 
\newtheorem{construction}[subsection]{Construction}

\theoremstyle{remark}
\newtheorem{remark}[subsection]{Remark}
\newtheorem{example}[subsection]{Example}

\title[Universal Central Extensions]{Universal Central Extensions\\ of Internal Crossed Modules\\ Via the Non-Abelian Tensor Product}
\author{Davide di Micco}
\address[Davide di Micco]{Università degli Studi di Milano, Via Saldini 50, 20133 Milano, Italy}
\email[Davide di Micco]{davide.dimicco@unimi.it}
\author{Tim Van~der Linden}
\address[Tim Van~der Linden]{Institut de Recherche en Math\'ematique et Physique, Universit\'e catholique de Louvain, che\-min du cyclotron~2 bte~L7.01.02, B--1348 Louvain-la-Neuve, Belgium}
\email[Tim Van~der Linden]{tim.vanderlinden@uclouvain.be}
\thanks{The second author is a Research Associate of the Fonds de la Recherche Scientifique--FNRS}
			
\date{\today}

\begin{document}
\begin{abstract}
In the context of internal crossed modules over a fixed base object in a given semi-abelian category, we use the non-abelian tensor product in order to prove that an object is perfect (in an appropriate sense) if and only if it admits a universal central extension. This extends results of Brown-Loday (\cite{BL87}, in the case of groups) and Edalatzadeh (\cite{Eda19}, in the case of Lie algebras). Our aim is to explain how those results can be understood in terms of categorical Galois theory: Edalatzadeh's interpretation in terms of quasi-pointed categories applies, but a more straightforward approach based on the theory developed in a pointed setting by Casas and the second author~\cite{CVdL14} works as well.
\end{abstract}



\maketitle

\section{Introduction}
The aim of this article is to study a result on universal central extensions of crossed modules due to Brown-Loday (\cite{BL87}, in the case of groups) and Edalatzadeh (\cite{Eda19}, in the case of Lie algebras). We prove, namely, that a crossed module over a fixed base object is perfect (in an appropriate sense) if and only if it admits a universal central extension. We first follow an ad-hoc approach, extending the result to the context of Janelidze-M\'arki-Tholen semi-abelian categories~\cite{JMT02} by using a general version, developed in~\cite{dMVdL19.3} of the non-abelian tensor product of Brown-Loday~\cite{BL87}. We then provide two interpretations from the perspective of categorical Galois theory. The first one follows the line of Edalatzadeh~\cite{Eda19} in the context of quasi-pointed~\cite{Bou01} categories (which have an initial object $0$ and a terminal object~$1$ such that $0\to 1$ is a monomorphism). This allows us to capture centrality, but we could not find a natural way to treat perfectness in this setting. We then switch to the pointed context ($0\cong 1$) where the theory developed by Casas and the second author~\cite{CVdL14} can be used. In this simpler environment we find a convenient interpretation both of centrality and perfectness.

The text is structured as follows. In Section~\ref{Section-2} we give an overview of basic definitions and results of categorical Galois theory, with particular emphasis on the example of so-called \emph{algebraically central extensions}. In Section~\ref{Section-3} we develop a more advanced example: the \emph{coinvariants reflection} from actions to trivial actions. A~key result here is Proposition~\ref{prop:birkhoff subcategory of actions}, which says that for any object $L$ a semi-abelian category, the trivial $L$-actions form a Birkhoff subcategory of the category of all $L$-actions. 

In Section~\ref{Section Internal crossed modules} we switch to the context of $L$-crossed modules. Here we recall basic aspects, results from commutator theory, the non-abelian tensor product, etc. In Section~\ref{Section Ad Hoc} we make an ad-hoc study of perfect objects and universal central extensions in this context, especially in relation to the tensor product. We prove our first main result, Theorem~\ref{thm:perfect iff admits a universal central extension}, which says that in any semi-abelian category satisfying the so-called \emph{Smith is Huq} condition \SH, an $L$-crossed module is perfect if and only if it admits a universal central extension. 

The last two sections of the article are devoted to two Galois-theoretic points of view on this result. In Section~\ref{Section Quasipointed} we consider a Galois theory in the quasi-pointed category $\XMod_L(\A)$ of $L$-crossed modules in $\A$, where we manage to give an interpretation of the central extensions (Theorem~\ref{cor:equivalent condition to centrality}). In Section~\ref{Section Pointed} we view an $L$-crossed module as an object of the semi-abelian category $\XMod(\A)$ and find a different Galois structure which characterises both the central extensions (Proposition~\ref{prop:equivalence between two concepts of universal central extension}) and the perfect objects (Proposition~\ref{prop:equivalence between two concepts of perfect object}).

\section{Revision of Galois theory and central extensions}\label{Section-2}
We recall some basic definitions and results of categorical Galois theory \cite{BJ01, Jan90, JK94, JK00}, especially in relation with algebraically central extensions.

A \emph{regular epimorphism} is a coequaliser of some pair of parallel arrows.

\begin{definition}
\label{defi:birkhoff subcategory}
Let $\C$ be an exact category and $\X$ a subcategory of $\C$. We say that $\X$ is a \emph{Birkhoff subcategory of $\C$} if the following hold:
\begin{enumerate}
\item $\X$ is a full and reflective subcategory of $\C$,
\item $\X$ is closed under subobjects in $\C$ and
\item $\X$ is closed under (regular epimorphic) quotients in $\C$.
\end{enumerate}
We usually denote the left adjoint as $I\colon\C\to\X$ and, when we do not omit it, the right adjoint as $H\colon\X\to\C$. The largest Birkhoff subcategory of $\C$ is obviously $\C$ itself, whereas the smallest one is given by $\Sub(1)$ where $1$ denotes the terminal object. When $\C$ is a variety, a Birkhoff subcategory is the same as a subvariety.
\end{definition}

\begin{lemma}[\cite{JK94}]
\label{lemma:equivalent conditions to Birkhoff}
A reflective subcategory $\X$ of an exact category $\C$ is a Birkhoff subcategory if and only if for each regular epimorphism $f\colon A\to B$, the naturality square
\begin{equation}\label{diag:Birk}
\vcenter{
\xymatrix{
A \ar[r]^-{\eta_A} \ar[d]_-{f} & HI(A) \ar[d]^-{HI(f)}\\
B \ar[r]_-{\eta_B} & HI(B)
}}
\end{equation}
is a pushout of regular epimorphisms.\noproof
\end{lemma}

Recall that a commutative square is a \emph{regular pushout} when all of its arrows, as well as the induced comparison to the pullback, are regular epimorphisms---see Figure~\ref{Figure Regular Pushout}. 
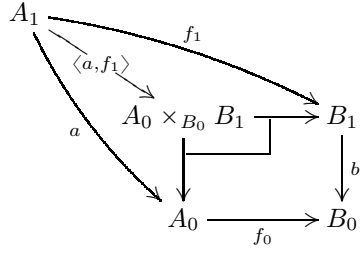
\begin{figure}
$\vcenter{\xymatrix{A_1\ar@/^/[rrd]^-{f_1} \ar@/_/[rdd]_-a \ar[rd]|-{\langle a,f_1\rangle} &&\\
&A_0\times_{B_0}B_1 \pullback\ar[r]\ar[d] &B_1\ar[d]^-b\\
&A_0\ar[r]_-{f_0} &B_0}}$
\caption{The outer square is a \emph{regular pushout} when all arrows in the induced diagram are regular epimorphisms.}\label{Figure Regular Pushout}
\end{figure}
In general, pushouts and regular pushouts do not coincide; by Theorem~5.7 in~\cite{CKP93} however, a regular category is an exact Mal'tsev category precisely when every pushout of two regular epimorphisms is a regular pushout. In particular, this is true in every semi-abelian category. 

\begin{lemma}\cite[Lemma 1.1]{BG02}
\label{lemma:ker of regular pushout is regular epimorphism}
In a semi-abelian category, consider a square $\beta\comp f=f'\comp \alpha$ of regular epimorphisms 
\[
\xymatrix{
K_f \ar@{.>}[d]_-{k} \ar[r]^-{k_f} & A \ar[r]^-{f} \ar[d]_-{\alpha} & B \ar[d]^-{\beta}\\
K_{f'} \ar[r]_-{k_{f'}} & A' \ar[r]_-{f'} & B'
}
\]
and take the kernels of $f$ and $f'$. The induced morphism $k$ is a regular epimorphism if and only if the given square is a regular pushout.\noproof
\end{lemma}

\subsection{Central extensions}
In the exact Mal'tsev context, for each Birkhoff subcategory there is a \emph{Galois theory}; we recall the main definitions having to do with central extensions.

\begin{definition}
We denote with $\Ext_B(\C)$ the \emph{category of extensions of $B$ in $\C$}, which is the full subcategory of $\C/B$ whose objects are the regular epimorphisms having $B$ as codomain; notice that a morphism in $\Ext_B(\C)$ is any triangle in $\C$ from a regular epimorphism to another regular epimorphism with the same codomain $B$.
\end{definition}

\begin{definition}
Given a Birkhoff subcategory $\X$ of $\C$ we say that an extension $f\colon A\to B$ is an \emph{$\X$-trivial extension (of $B$)} when the naturality square \eqref{diag:Birk} is a pullback in $\C$. We will denote with $\Triv_B(\C,\X)$ the full subcategory of $\Ext_B(\C)$ whose objects are the $\X$-trivial extensions of $B$.
\end{definition}

\begin{definition}
\label{defi:central extensions}
Given a Birkhoff subcategory $\X$ of $\C$ we say that an extension $f\colon A\to B$ is an \emph{$\X$-central extension (of $B$)} when there exists an extension $g\colon C\to B$ such that the pullback $g^*(f)$
\[
\xymatrix{
A\times_B C \pullback \ar[r]^-{g^*(f)} \ar[d] & C \ar[d]^-{g}\\
A \ar[r]_-{f} & B
}
\]
of $f$ along $g$ is an $\X$-trivial extension.
We will denote by $\Centr_B(\C,\X)$ the full subcategory of $\Ext_B(\C)$ whose objects are the $\X$-central extensions of $B$. We have the chain of inclusions
\[
\Triv_B(\C,\X)\subseteq \Centr_B(\C,\X)\subseteq \Ext_B(\C).
\]
\end{definition}

\begin{lemma}
\label{lemma:equivalent conditions to central}
In the context of an exact protomodular category $\C$, an extension $f\colon A\to B$ is $\X$-central if and only if any of the projections $\pi_1$, $\pi_2$ in the kernel pair
\[
\KP(f) =(A\times_BA,\pi_1,\pi_2)
\qquad\qquad
\vcenter{\xymatrix{
A\times_BA \pullback \ar[r]^-{\pi_2} \ar[d]_-{\pi_1} & A \ar[d]^-{f} \\
A \ar[r]_-{f} & B
}}
\]
is an $\X$-trivial extension. Furthermore, a split epimorphism is an $\X$-central extension if and only if it is $\X$-trivial.
\end{lemma}
\begin{proof}
Proposition~$4.7$ in~\cite{JK94} tells us that the two claims are equivalent while Theorem~$4.8$ proves that they hold in every Goursat category. Protomodularity implies the Mal'tsev property, which is stronger than the Goursat property.
\end{proof}

\subsection{Example: algebraically and categorically central extensions}
\label{higgins commutator and central extensions}
A key example of a Birkhoff subcategory is the subcategory $\Ab(\A)$ of abelian objects in any semi-abelian category $\A$, which are those objects that admit an internal abelian group structure. For instance, abelian groups in the category of all groups, or vector spaces equipped with a trivial (zero) multiplication in any category of Lie algebras over a field. It is clear that $\Ab(\A)$ is an abelian category, but it is also a Birkhoff subcategory of $\A$: indeed it is a full reflective subcategory of $\A$, closed under subobjects and regular quotients. 

This means that we have a definition of $\Ab(\A)$-central extensions, also called \emph{categorically central extensions} in contrast with \emph{algebraically central extensions}: the former ones are given through Definition~\ref{defi:central extensions}, whereas the latter ones arise naturally from commutator theory. As it turns out, the two types of central extensions coincide~\cite{BG02}; let us elaborate on this. 

Consider a cospan $(k\colon K\to X, l\colon L\to X)$. By definition~\cite{MM10,HVdL11,Hig56}, the \emph{Higgins commutator} $[K,L]\leq X$ is computed as in the commutative diagram
\[
\xymatrix{
 0 \ar[r] & K \cosmash L \ar[r]^{\iota_{K,L}} \ar[d] & K + L \ar[r]^{r_{K,L}} \ar[d]^-{\binom{k}{l}} 
 & K \times L \ar[r] & 0 \\
 & [K,L] \ar[r] & X
}
\]
where $r_{K,L}=\bigl\lgroup\begin{smallmatrix}
	1_K & 0 \\ 0 & 1_L
\end{smallmatrix}\bigr \rgroup$ is the canonical morphism from the coproduct to the product, $\iota_{K,L}$ is its kernel and $[K,L]$ is the image of the composite $\binom{k}{l}\comp \iota_{K,L}$. The object $K\cosmash L$ is called the \emph{co-smash product}~\cite{CJ03} of $K$ and $L$. 

The object $X$ is abelian if and only if $[X,X]=0$, so that the reflector ${\A\to\Ab(\A)}$ sends an object $X$ to the quotient $X/[X,X]$. 

Theorem~6.3 in~\cite{MM10} says that in a semi-abelian category, a subobject $K\leq X$ is normal (we write $K\lhd L$) if and only if $[K,X]\leq K$. Proposition~4.14 in~\cite{HL13} adds to this that the normal closure $\cl_X(K)$ of $K$ in $X$---a priori, the kernel of the cokernel of a representing monomorphism $K\to X$---may be obtained as the join $K\vee [K,X]\lhd X$. 

\emph{Algebraic centrality} of an extension $f\colon A\to B$ is now the condition that $[K,A]$ is trivial, for $K$ the kernel of $f$. Equivalently, by \cite{GVdL08} combined with~\cite{Gra02,JK00}, this may be expressed in terms of the kernel pair of $f$. The main result of~\cite{BG02} says that algebraically and categorically central extensions coincide. This implies right away that in the cases of groups and Lie algebras, we regain the classical definitions.

\subsection{Universal central extensions}
The following definitions are borrowed from the article~\cite{CVdL14}, where the theory of universal central extensions is explored in detail. We consider a Birkhoff subcategory $\X$ of a pointed exact Mal'tsev category $\C$.

\begin{definition}
We say that an extension $u\colon U\to B$ is a \emph{universal $\X$-central extension of $B$} if it is an initial object in $\Centr_B(\C,\X)$.
\end{definition}

\begin{definition}
\label{defi:perfect object in pointed setting}
We say that an object $A\in\C$ is \emph{$\X$-perfect} whenever its reflection $I(A)$ is the zero object $0\in\X$.
\end{definition}

Via the analysis in~\ref{higgins commutator and central extensions}, these definitions capture the usual ones for groups and Lie algebras. A key result in this general context is~\cite[Theorem~3.5]{CVdL14}, which says that an object in a semi-abelian category $\A$ is perfect with respect to a Birkhoff subcategory $\X$ of $\A$ if and only if it admits a universal $\X$-central extension.

\subsection{More on trivial extensions}\label{subsec:TrivialExtensionReflection}
Given a Birkhoff subcategory $\X$ of an exact Mal'tsev category $\C$, it is well known and easy to see that the category $\Triv_B(\C,\X)$ is again reflective in $\Ext_B(\C)$: reflect the given extension into $\X$, then pull back along the unit. In the setting~\ref{higgins commutator and central extensions} of a semi-abelian category $\A$ with its Birkhoff subcategory of abelian objects $\Ab(\A)$, we may restrict the left adjoint to the split epimorphisms in $\A$, and via the results in~\cite{EverVdL1} find the following. Let $f\colon A\to B$ be a split epimorphism, with splitting $s$ and kernel $k\colon K\to A$; then the unit of the adjunction at $f$ gives rise to the morphism of split short exact sequences 
\[
\xymatrixcolsep{4pc}
\xymatrix{
0 \ar[r] & K \ar[d] \ar[r]^-k & A \ar[d] \ar@<.5ex>[r]^-{f} & B \ar@{=}[d] \ar@<.5ex>[l]^-{s} \ar[r] & 0\\
0 \ar[r] & \frac{K}{[K,A]} \ar[r]_-{\langle 1_{{K}/{[K,A]}},0\rangle} & \frac{K}{[K,A]}\times B \ar@<.5ex>[r]^-{\pi_2} & B \ar@<.5ex>[l]^-{\langle 0,1_B\rangle} \ar[r] & 0.
}
\]
Note, in particular, that the object $K/[K,A]$ is abelian.

\section{Actions, trivial actions, coinvariants}\label{Section-3}
In this section we work out a less trivial example of a Galois structure, which later on will be useful for us: we study the so-called \emph{coinvariants reflector} from internal actions to trivial actions. This is a categorical conceptualisation of a classical construction, well known in group cohomology: see~\cite{Br82}, for instance. It generalises the result of~\ref{subsec:TrivialExtensionReflection} to split extensions with a non-abelian kernel.

We start by recalling some well-known basic results on limits and colimits, easily checked by hand, in the category of points over a fixed base object $L$ in a given semi-abelian category $\A$.

\subsection{Points, actions, split extensions}
A \emph{point} $(p,s)$ in a category $\A$ is a split epimorphism $p\colon X\to L$ together with a chosen splitting $s\colon L\to X$, so that $p\comp s=1_L$. The category $\Pt(\A)$ of \emph{points in $\A$} has, as objects, points in $\A$, and as morphisms, natural transformations between such. Fixing the object $L$, we obtain the subcategory $\Pt_L(\A)$ of \emph{points over $L$}. If $\A$ is a semi-abelian category, then a point $(p,s)$ with a chosen kernel $k$ of $p$ is the same thing as a \emph{split extension} in~$\A$: a split short exact sequence
\[
\xymatrix{0 \ar[r] & K \ar[r]^{k} & X \ar@<.5ex>[r]^-p & L \ar@<.5ex>[l]^-s \ar[r] & 0,}
\]
which means that $k$ is the kernel of $p$, that $p$ is the cokernel of $k$, and that $p\comp s=1_L$. In such a split extension, $k$ and $s$ are jointly extremal-epimorphic. 
Via a semi-direct product construction~\cite{BJ98}, we have an equivalence $\Pt_L(\A)\simeq\Act_L(\A)$, where the latter category of \emph{internal $L$-actions} in $\A$ consists of the algebras of the monad $(L\flat (-),\eta^L,\mu^L)$ defined through 
\[
\xymatrix{0 \ar[r]& L\flat M\ar[r]^-{k_{L,M}} & L+M \ar@<.5ex>[r]^-{\binom{1_L}{0}} & L\ar[r] \ar@<.5ex>[l]^-{\iota_L} & 0.}	
\]
One functor in the equivalence sends a point $(p,s)$ to the action $(L,K,\xi)$ in
\begin{align*}
\xymatrix{
0 \ar[r] & L\flat K \pullback \ar@{.>}[d]_-{\xi} \ar[r]^-{k_{L,K}} & L+K \ar[d]^-{\binom{s}{k}} \ar[r]^-{\binom{1_L}{0}} & L\ar@{=}[d] \ar[r] & 0\\
0 \ar[r] & K \ar[r]_-{k} & X \ar[r]_-{p} & L\ar[r] & 0.
}
\end{align*}
The other functor sends an action $(L,M,\xi)$ to the induced semidirect product, which is the point $({\pi_\xi\colon M\rtimes_{\xi} L \to L}, {i_\xi\colon L\to M\rtimes_{\xi} L })$, where $M\rtimes_{\xi} L$ is the coequaliser
\[
\xymatrixcolsep{4pc}
\xymatrixrowsep{3pc}
\xymatrix{
L\flat M \ar@<4pt>[r]^-{i_M\circ\xi} \ar@<-1pt>[r]_-{k_{L,M}} & L+M \ar[r]^{\sigma_{\xi}} & M\rtimes_{\xi} L,
}
\]
the morphism $\pi_\xi\colon M\rtimes_{\xi} L\to L$ is the unique morphism such that $\binom{1_L}{0}=\pi_{\xi}\comp \sigma_{\xi}$, and finally $i_\xi=\sigma_\xi\comp i_L$. 
We will denote $M\rtimes_{\xi} L$ as $M\rtimes L$ if there is no risk of confusion regarding the action involved. The morphism $k_{\xi}\coloneqq\sigma_{\xi}\comp i_M\colon M\to M\rtimes_{\xi} L $ is always the kernel of $\pi_{\xi}$: it is easy to see that $\pi_{\xi}\comp k_{\xi}=0$, whereas for the universal property some work needs to be done.

All equivalences $\Pt_L(\A)\simeq\Act_L(\A)$ taken together determine an equivalence $\Pt(\A)\simeq\Act(\A)$, where $\Act(\A)$ is the category of \emph{internal actions in~$\A$}; its morphisms are suitably defined equivariant maps.

\begin{lemma}
\label{lemma:regular epimorphisms of L-points}
A morphism in the category $\Pt_L(\A)$ is a regular epimorphism if and only if the morphism between the domains is a regular epimorphism in $\A$.\noproof
\end{lemma}

\begin{lemma}
\label{lemma:pushouts and pullbacks of L-points}
A square in the category $\Pt_L(\A)$ is a pushout (a pullback) if and only if the square between the domains is a pushout (a pullback) in $\A$. This means that pushouts and pullbacks can be computed in the base category using just the domains: the additional structure is canonically induced.\noproof
\end{lemma}

Since the point $(1_L,1_L)$ is zero in $\Pt_L(\A)$, as a consequence we have a simple way to compute kernels in that category.

\begin{corollary}
\label{cor:kernel of L-point morphism}
Consider a morphism of $L$-points as in Figure~\ref{FigKernelPoints} on the left. 
\begin{figure}
\[
\vcenter{\xymatrix{
A \ar[d]_-{f} \ar@<.5ex>[r]^-{p_A} & L \ar@{=}[d] \ar@<.5ex>[l]^-{s_A}\\
B \ar@<.5ex>[r]^-{p_B} & L \ar@<.5ex>[l]^-{s_B}
}}
\qquad\qquad
\vcenter{\xymatrix@!0@=4em{L\times_B A \dottedpullback \ar@<.5ex>@{.>}[rd] \ar@{.>}[rr] \ar@{.>}[dd] && A \ar@<.5ex>[ld]^-{p_A} \ar[dd]^f\\
& L \ar@<.5ex>[rd]^-{s_B} \ar@<.5ex>[ru]^-{s_A} \ar@<.5ex>[ld]^-{1_L} \ar@<.5ex>@{.>}[lu] \\
L \ar[rr]_{s_B} \ar@<.5ex>[ru]^-{1_L} && B \ar@<.5ex>[lu]^-{p_B}}}
\]
\caption{Kernels in $\Pt_L(\A)$.}\label{FigKernelPoints}
\end{figure}
Then its kernel is the $L$-point induced by the outer pullback on the right.\noproof
\end{corollary}

\subsection{The join decomposition formula}
In what follows, we need to be able to decompose a commutator of a join of subobjects into a join of commutators. This goes by means of a join decomposition formula which involves a ternary version of the Higgins commutator:

\begin{definition}[\cite{CJ03,HVdL11,HL13}]
Given objects $A$, $B$ and $C$ in $\A$, consider the morphism
\[
\Sigma_{A,B,C}=
\begin{pmatrix}
i_A & i_A & 0 \\
i_B & 0 & i_B \\
0 & i_C & i_C
\end{pmatrix}
\colon A+B+C \longrightarrow (A+B)\times(A+C)\times(B+C)
\]
and its kernel $h_{A,B,C}\colon A\diamond B\diamond C \to A+B+C$. The object $A\diamond B\diamond C$ is called the \emph{cosmash product} of $A$, $B$ and $C$. 

Given three subobjects $(K,k)$, $(M,m)$ and $(N,n)$ of an object $X$, we define their \emph{Higgins commutator} as the subobject of $X$ given by the image factorisation 
\[
\xymatrix{
K\diamond M\diamond N \ar@{.>}[d] \ar[r]^-{h_{K,M,N}} & K+M+N \ar[d]^-{\Bigl\lgroup\begin{smallmatrix}
	k \\ m\\ n
\end{smallmatrix}\Bigr\rgroup}\\
[K,M,N] \ar@{.>}[r] & X.
}
\]
We call $[K,M,N]$ the \emph{ternary Higgins commutator} of $K$, $M$ and $N$ in $X$.
\end{definition}

\begin{proposition}[\cite{HVdL11, HL13}]\label{Higgins properties}
Suppose $K_1$, $K_2$, $K_3\leq X$. Then we have the following (in)equalities of subobjects of $X$:
\begin{enumerate}
\item[(0)] if $K_1=0$ then $[K_1,K_2]=0=[K_1,K_2,K_3]$;
\item $[K_1,K_2]=[K_2,K_1]$ and for $\sigma\in S_3$, $[K_1,K_2,K_3]=[K_{\sigma(1)},K_{\sigma(2)},K_{\sigma(3)}]$;
\item for any regular epimorphism $f\colon {X\to Y}$, $f[K_1,K_2 ]=[f(K_1),f(K_2)]\leq Y$ and $f[K_1,K_2,K_3 ]=[f(K_1),f(K_2),f(K_3)]\leq Y$;
\item $[L_1,K_2]\leq [K_1,K_2]$ and $[L_1,K_2,K_3]\leq [K_1,K_2,K_3]$ when $L_1\leq K_1$;
\item $[[K_1,K_{2}],K_{3}]\leq [K_1,K_2,K_{3}]$;
\item $[K_1,K_1,K_2]\leq [K_1,K_2]$;
\item $[K_1,K_{2}\vee K_3]=[K_1,K_2]\vee [K_1,K_3]\vee [K_1,K_{2},K_3]$.
\end{enumerate}
\end{proposition}

As we shall see, item (6) allows us to reduce commutators to simpler ones.

\subsection{Trivial actions}
We now define a suitable Birkhoff subcategory of $\Act_L(\A)$: the subcategory $\TrivAct_L(\A)$ of trivial $L$-actions.

\begin{definition}
Consider an $L$-action expressed as a point with a chosen kernel
\[
\xymatrix{
0 \ar[r] & M \ar[r]^-{k_p} & X \ar@<.5ex>[r]^-{p} & L \ar@<.5ex>[l]^-{s} \ar[r] & 0.
}
\]
We say that it is a \emph{trivial action} when there exists an isomorphism of split short exact sequences
\[
\xymatrixcolsep{4pc}
\xymatrix{
0 \ar[r] & M \ar@{=}[d] \ar[r]^-{k_p} & X \ar[d]^-{\phi}_{\cong} \ar@<.5ex>[r]^-{p} & L \ar@{=}[d] \ar@<.5ex>[l]^-{s} \ar[r] & 0\\
0 \ar[r] & M \ar[r]_-{\langle 1_M,0\rangle} & M\times L \ar@<.5ex>[r]^-{\pi_2} & L \ar@<.5ex>[l]^-{\langle 0,1_L\rangle} \ar[r] & 0.
}
\]
The category $\TrivAct_L(\A)$ of trivial $L$-actions is the full subcategory of $\Act_L(\A)$ whose objects are trivial $L$-actions.
\end{definition}

\begin{construction}\label{ConstructionTrivialAction}
We wish to construct a functor $I\colon \Act_L(\A)\to \TrivAct_L(\A)$, left adjoint to the inclusion functor $H\colon \TrivAct_L(\A)\to \Act_L(\A)$. Given a split extension as in the top row of the diagram
\[
\xymatrixcolsep{4pc}
\xymatrix{
0 \ar[r] & M \ar[d]_-{c_s\circ k_p} \ar[r]^-{k_p} & X \ar[d]|-{\langle c_s,p\rangle} \ar@<.5ex>[r]^-{p} & L \ar@{=}[d] \ar@<.5ex>[l]^-{s} \ar[r] & 0\\
0 \ar[r] & C_s \ar[r]_-{\langle 1_{C_s},0\rangle} & C_s\times L \ar@<.5ex>[r]^-{\pi_2} & L \ar@<.5ex>[l]^-{\langle 0,1_L\rangle} \ar[r] & 0,
}
\]
we take the cokernel $C_s$ of the splitting $s$, which leads to the bottom split extension and the morphism between them. The trivial $L$-action corresponding to the bottom sequence is called the \emph{object of coinvariants} of the given action, and it is the image through $I$ of the action we began with. The morphism of split extensions corresponds to the unit $\eta\colon 1_{\Act_L(\A)}\to HI$ of the adjunction at $(p,s)$. 
\end{construction}

We still need to prove that the thus constructed functor $I$ is a Birkhoff reflector, of course. The following definition follows the pattern of~\cite{EverVdL1, EverVdLRCT}: the kernel of the unit of a Birkhoff reflector is viewed as a commutator, relative to this reflector.

\begin{definition}
\label{defi:coinvariance commutator}
With the notation of the previous construction, in Figure~\ref{FigKernelUnit} we take the kernel of the unit $\eta_{(p,s)}$ as in Corollary~\ref{cor:kernel of L-point morphism}. 
\begin{figure}
\resizebox{.4\textwidth}{!}
{$\xymatrix@!0@=7em{\cl_X(L) \dottedpullback \ar@<.5ex>@{.>}[rd] \ar@{.>}[rr]^-{k_{c_s}} \ar@{.>}[dd]_-{p\circ k_{c_s}} && X \ar@<.5ex>[ld]^-{p} \ar[dd]^-{\langle c_s,p\rangle}\\
& L \ar@<.5ex>[rd]^-{\langle 0,1_L\rangle} \ar@<.5ex>[ru]^-{s} \ar@<.5ex>[ld]^-{1_L} \ar@<.5ex>@{.>}[lu] \\
L \ar[rr]_-{\langle 0,1_L\rangle} \ar@<.5ex>[ru]^-{1_L} && C_s\times L \ar@<.5ex>[lu]^-{\pi_2}}
$}
\caption{Kernel of the unit of the adjunction.}\label{FigKernelUnit}
\end{figure}
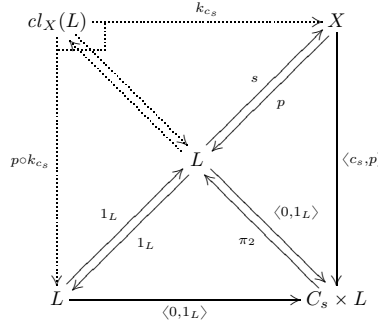
It is easily seen that this square is indeed a pullback. Recall that $\cl_X(L)$ is the normal closure of $L\leq X$ in $X$, which may be obtained as the kernel of $c_s$. Taking kernels of the split epimorphisms, we obtain horizontal short exact sequences as in
\[
\xymatrixcolsep{4pc}
\xymatrix{
0 \ar[r] & \llbracket L,M\rrbracket \pullback \ar[d]_-{k_{(c_s\circ k_p)}} \ar[r] & \cl_X(L) \ar[d]^-{k_{c_s}} \ar@<.5ex>[r] & L \ar@{=}[d] \ar@<.5ex>[l] \ar[r] & 0\\
0 \ar[r] & M \pullback \ar[d]_-{c_s\circ k_p} \ar[r]^-{k_p} & X \ar[d]|-{\langle c_s,p\rangle} \ar@<.5ex>[r]^-{p} & L \ar@{=}[d] \ar@<.5ex>[l]^-{s} \ar[r] & 0\\
0 \ar[r] & C_s \ar[r]_-{\langle 1_{C_s},0\rangle} & C_s\times L \ar@<.5ex>[r]^-{\pi_2} & L \ar@<.5ex>[l]^-{\langle 0,1_L\rangle} \ar[r] & 0
}
\]
and we define the \emph{coinvariants commutator} $\llbracket L,M\rrbracket$ as the top left kernel. 
\end{definition}

\begin{remark}
\label{rmk:eta is regular epimorphism}
By construction we have the diagram
\[
\xymatrixcolsep{4pc}
\xymatrix{
0 \ar[r] & M \ar[d]_-{c_s\circ k_p} \ar[r]^-{k_p} & X \ar[d]|-{\langle c_s,p\rangle} \ar@<.5ex>[r]^-{p} & L \ar@{=}[d] \ar@<.5ex>[l]^-{s} \ar[r] & 0\\
0 \ar[r] & C_s \ar[r]_-{\langle 1_{C_s},0\rangle} \ar@{=}[d] & C_s\times L \ar@<.5ex>[r]^-{\pi_2} \ar[d]_-{\pi_1} & L \ar@<.5ex>[l]^-{\langle 0,1_L\rangle} \ar[r] \ar[d] & 0\\
0 \ar[r] & C_s \ar@{=}[r] & C_s \ar@<.5ex>[r] & 0 \ar@<.5ex>[l] \ar[r] & 0
}
\]
where the vertical composite rectangle
\[
\xymatrix{
X \ar[r]^{p} \ar[d]_{c_s} & L \ar[d]^{0}\\
C_s \ar[r]_{0} & \pushout 0 
}
\]
is a pushout of regular epimorphisms, hence a regular pushout. Indeed the universal property can be shown directly by using the fact that $p\comp s=1_L$ and that $c_s$ is the cokernel of $s$. Since $\langle c_s,p\rangle$ is the comparison morphism to the induced pullback, it is automatically a regular epimorphism. By Lemma~\ref{lemma:regular epimorphisms of L-points}, this is equivalent to~$\eta_{(p,s)}$ being a regular epimorphism of points over $L$. Furthermore, since the top left square is a pullback, also $c_s\comp k_p$ is a regular epimorphism.
\end{remark}

\begin{remark}
Since kernels commute with kernels, we can obtain $\llbracket L,M\rrbracket$ as the kernel of $c_s\comp k_p$, computed in $\A$. Since the lower left square in the diagram of Definition~\ref{defi:coinvariance commutator} is a pullback, the composite $k_p\comp k_{(c_s\circ k_p)}$ is the kernel of $\langle c_s,p\rangle$, so that $\llbracket L,M\rrbracket\lhd X$. On the other hand, since the upper left square is a pullback as well, we have that $\llbracket L,M\rrbracket= M\wedge \cl_X(L)$. 

An alternative argument goes as follows. $M$ is the kernel of $p$, while the kernel of~$c_s$ is precisely the normal closure of $L$ in $X$; the kernel of ${\langle c_s,p\rangle}$ is the intersection of those two kernels. 
\end{remark}

By the discussion in~\ref{higgins commutator and central extensions}, we know that $\cl_X(L)=L\vee [L,X]$ in $X$. On the other hand, the top split extension in the diagram of Definition~\ref{defi:coinvariance commutator} tells us that $\cl_X(L)=L\vee \llbracket L,M\rrbracket$. The following simplifies this, by relating the two types of commutator.

\begin{proposition}
\label{prop:higgins=coinvariance}
Given a split extension over $L$
\begin{equation}
\label{diag:a point}
\xymatrix{
0 \ar[r] & M \ar[r]^-{k} & X \ar@<.5ex>[r]^-{p} & L \ar@<.5ex>[l]^-{s} \ar[r] & 0,
}
\end{equation}
its coinvariance commutator $\llbracket L,M\rrbracket$, seen as a subobject of $X$, coincides with the Higgins commutator $[L,M]$ of $L$ and $M$ in $X$. In particular, $C_s\cong {M}/{\left[L,M\right]}$.
\end{proposition}
\begin{proof}
Consider the morphism of split extensions 
\[
\xymatrix{
0 \ar[r] & L\diamond M \ar@{.>}[d] \ar[r]^-{i_{L,M}} & L\flat M \ar[d]|-{\binom{s}{k}\circ \kappa_{L,M}} \ar@<.5ex>[r] & L \ar@{=}[d] \ar@<.5ex>[l] \ar[r] & 0\\
0 \ar[r] & M \ar[r]_-{k} & X \ar@<.5ex>[r]^-{p} & L \ar@<.5ex>[l]^-{s} \ar[r] & 0.
}
\]
Its image is the point $L\vee [L,M]\rightleftarrows L$, whose kernel is $[L,M]$. 

Both $\cl_X(L)\rightleftarrows L$ and $L\vee [L,M]\rightleftarrows L$ are normal subobjects of $X\rightleftarrows L$. (For the latter, this follows because $\binom{s}{k}$ is a regular epimorphism and $\kappa_{L,M} \comp i_{L,M}$ is a normal monomorphism in $\A$, so that the image $[L,M]$ of their composite is normal in $X$.) Hence if we show that one vanishes if and only if the other does---so that they express the same universal property, namely the condition that \eqref{diag:a point} represents a trivial action---then they coincide. For the point $\cl_X(L)\rightleftarrows L$ we already know that its kernel is $\llbracket L,M\rrbracket$, which is zero if and only if \eqref{diag:a point} is trivial.

First suppose that \eqref{diag:a point} represents a trivial action, so that $\llbracket L,M\rrbracket=0$. Then $M$ and $\cl_X(L)$ are two normal subobjects of $X$ with a zero intersection, which implies that $[\cl_X(L),M]\leq \cl_X(L)\wedge M$ is trivial. Hence $[L,M]\leq [\cl_X(L),M]=0$.

Conversely, if $[L,M]=0$, then by Proposition~\ref{Higgins properties},
\[
[L,X]=[L,M]\vee [L,L]\vee [L,L,M]\leq [L,M]\vee [L,L]=[L,L]\leq L
\]
so that $\cl_X(L)=L$ and $\llbracket L,M\rrbracket$ vanishes. 
\end{proof}

\begin{proposition}
\label{prop:birkhoff subcategory of actions}
$\TrivAct_L(\A)$ is a Birkhoff subcategory of $\Act_L(\A)$.
\end{proposition}
\begin{proof}
By its construction and by Remark~\ref{rmk:eta is regular epimorphism}, the regular epimorphic natural transformation $\eta\colon 1_{\Act_L(\A)}\to HI$ is the cokernel of a normal subfunctor $V$ of $1_{\Act_L(\A)}$ which sends a split extension as in~\eqref{diag:a point} to the split extension determined by the point $\cl_X(L)\rightleftarrows L$. By Corollary~5.7 in~\cite{EverVdL1}, this shows that $I$ is a Birkhoff reflector if and only if the functor $V$ preserves regular epimorphisms. Hence it suffices to prove that each regular epimorphism of points over $L$
 \[
 \xymatrix{
 M \ar@<.5ex>[r]^-{p} \ar[d]_-{f} & L \ar@{=}[d] \ar@<.5ex>[l]^-{s}\\
 M' \ar@<.5ex>[r]^-{p'} & L \ar@<.5ex>[l]^-{s'}
 }
 \]
is sent to a regular epimorphism of points by the functor $V$. By Proposition~\ref{prop:higgins=coinvariance}, this holds because the induced morphism $[1_L,f]\colon[L,M]\to [L,M']$ is again a regular epimorphism; this follows from Lemma~5.11 in~\cite{MM10}.
\end{proof}

\begin{remark}
\label{rmk:perfect L-actions}
According to Definition~\ref{defi:perfect object in pointed setting} we have that an $L$-action corresponding to~\ref{rmk:eta is regular epimorphism} is $\TrivAct_L(\A)$-perfect if and only if its image through $I$ is the zero $L$-action $(0\colon {0\to L}, \tau^L_0)$ which corresponds to the split extension
\[
\xymatrix{
0 \ar[r] & 0 \ar[r] & L \ar@<.5ex>[r]^-{1_L} & L \ar@<.5ex>[l]^-{1_L} \ar[r] & 0.
}
\]
This, in turn, is equivalent to the equality of subobjects $\llbracket L,M\rrbracket=[L,M]=M$. Hence an $L$-action on an object $M$ is perfect if and only if $M\leq \cl_X(L)$, which is equivalent to saying that the normal closure $\cl_X(L)$ of $L$ in $X$ is all of $X$.
\end{remark}

\section{Internal crossed modules}\label{Section Internal crossed modules}

We now focus on internal crossed modules in semi-abelian categories. Internal crossed modules are equivalent to internal categories; the conditions that make this happen were obtained in~\cite{Jan03}. In order to have a description which is as simple as possible, we require that $\A$ satisfies an additional condition, called the \emph{Smith is Huq condition} \SH. A semi-abelian category satisfies it when the Smith/Pedicchio commutator~\cite{Ped95b} of two internal equivalence relations vanishes if and only if so does the Huq commutator of their associated normal subobjects~\cite{BB04,MFVdL12}. As explained in~\cite{HVdL11}, in terms of Higgins commutators, this amounts to the condition that whenever $M$, $N\lhd L$ are normal subobjects, $[M,N]=0$ implies $[M,N,L]=0$. 

Examples of semi-abelian categories that satisfy \SH\ include the categories of groups, (commutative) rings (not necessarily unitary), Lie algebras over a commutative ring with unit, Poisson algebras and associative algebras, as are all varieties of such algebras, and crossed modules over those. In fact, all \emph{Orzech categories of interest}~\cite{Orz72,CGVdL15b} are examples. On the other hand, the category of loops is semi-abelian but does not satisfy \SH. Further details can be found in~\cite{Jan03,HVdL11,MFVdL12}. 

The work of Janelidze~\cite{Jan03} provides an explicit description of internal crossed modules in terms of internal actions, together with an equivalence of categories $\XMod(\A)\simeq \Grpd(\A)$ which extends the equivalence $\Act(\A)\simeq \Pt(\A)$. Since the category of internal groupoids in a semi-abelian category is again semi-abelian~\cite{BG02}, the category of internal crossed modules is semi-abelian as well. It is explained in~\cite{HVdL11} that under \SH, Higgins commutators suffice for the description of internal groupoids. Furthermore, the characterisation of internal crossed modules given in~\cite{Jan03} simplifies---see below. This is our main reason for working in this context. 

\begin{definition}
\label{defi:internal crossed modules}
In a semi-abelian category $\A$ with \SH, an \emph{internal crossed module} is a pair $(\partial\colon{M\to L},\xi)$ where $\partial\colon{M\to L}$ is a morphism in $\A$ and $\xi\colon{L\flat M\to M}$ is an internal action such that the diagram 
\[
\xymatrix{
M\flat M \ar[d]_-{\chi_M} \ar[r]^-{\partial\flat 1_M} \ar@{}[rd]|-{(*_1)} & L\flat M \ar@{}[rd]|-{(*_2)}\ar[d]^-{\xi} \ar[r]^-{1_L\flat\partial} & L\flat L \ar[d]^-{\chi_L} \\
M \ar@{=}[r] & M\ar[r]_-{\partial} & L
}
\]
commutes. $(*_1)$ is the \emph{Peiffer condition}, and $(*_2)$ the \emph{precrossed module condition}.
\end{definition}

In this general context we have been able to define, for each pair of coterminal internal crossed modules $(\mu\colon{M\to L},\xi_M)$ and $(\nu\colon{N\to L},\xi_N)$, a generalisation of the Brown-Loday non-abelian tensor product: see Figure~\ref{Figure Intro Tensor}. 
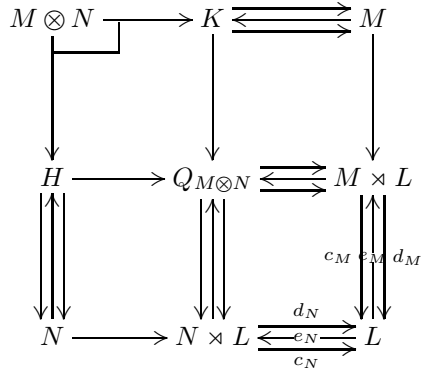
\begin{figure}
\[
\xymatrix@!0@=6em{
M\otimes N \pullback \ar[r] \ar[d] & K \ar@<1ex>[r] \ar@<-1ex>[r] \ar[d] & M \ar[l] \ar[d]\\
H \ar@<1ex>[d] \ar@<-1ex>[d] \ar[r] & Q_{M\otimes N} \ar@<1ex>[d] \ar@<-1ex>[d] \ar@<1ex>[r] \ar@<-1ex>[r] & M\rtimes L \ar[l] \ar@<1ex>[d]^-{d_M} \ar@<-1ex>[d]_-{c_M}\\
N \ar[u] \ar[r] & N\rtimes L \ar[u] \ar@<1ex>[r]^-{d_N} \ar@<-1ex>[r]_-{c_N} & L \ar[u]|-{e_M} \ar[l]|-{e_N}
}
\]
\caption{The tensor product as normalisation of a double groupoid.}\label{Figure Intro Tensor}
\end{figure}
In particular, our construction uses the universal property of the non-abelian tensor product described in~\cite{BL87}, through the equivalence $\Grpd(\A)\simeq\XMod(\A)$ between the categories of internal groupoids and internal crossed modules. For further details see~\cite{dMVdL19.3} and the proof of Proposition~\ref{prop:delta_M is a regular epimorphism}.

\begin{lemma}
\label{lemma:regular epimorphisms in XMOD(A)}
Consider a morphism of internal crossed modules
\[
 (\partial\colon{M\to L},\xi)\xlongrightarrow{(f,l)}(\partial'\colon{M'\to L'},\xi').
\]
Then $(f,l)$ is a regular epimorphism in $\XMod(\A)$ if and only $f$ and $l$ are regular epimorphisms in $\A$.
\end{lemma}
\begin{proof}
In the category $\RG(\A)$ of reflexive graphs in $\A$, coequalisers are computed pointwise, and due to Theorem~$3.1$ and Lemma~$3.1$ in~\cite{Gra99} this implies that also in $\Grpd(\A)$ the coequalisers are computed pointwise. This means that a morphism 
\[
(X,L,d,c,e,m)\xlongrightarrow{(x,l)}(X',L',d',c',e',m')
\]
is the coequaliser of $(g_0,g_1)$ and $(h_0,h_1)$ in $\Grpd(\A)$ if and only if $l$ is the coequaliser $c_{g_1,h_1}$ and if $x$ is the coequaliser $c_{g_0,h_0}$. Using the equivalence of categories $\XMod(\A)\cong \Grpd(\A)$ and the diagram 
\[
\xymatrixcolsep{4pc}
\xymatrixrowsep{3pc}
\xymatrix{
M \ar[d]_-{f} \ar[r]^-{k_{d}} & X \ar[d]^-{x} \ar@<1ex>[r]^-{d} \ar@<-1ex>[r]_-{c} & L \ar[d]^-{l} \ar[l]|-{e}\\
M' \ar[r]_-{k_{d'}} & X' \ar@<1ex>[r]^-{d'} \ar@<-1ex>[r]_-{c'} & L' \ar[l]|-{e'}
}
\]
where $X=M\rtimes_{\xi}L$, $X'=M'\rtimes_{\xi'}L$ and $x=f\rtimes l$, we conclude that $(f,l)$ is a regular epimorphism in $\XMod(\A)$ if and only if both $l$ and $x$ are regular epimorphisms in~$\A$. Now it suffices to apply the ``Short Five Lemma for regular epimorphisms''---item~5 in \cite[Lemma~4.2.5]{BB04}---to finish the proof.
\end{proof}

\begin{lemma}
\label{lemma:M over the commutator is abelian}
For any internal crossed module $(\partial\colon{M\to L},\xi)$, the object ${M}/{\left[L,M\right]}$ is abelian.
\end{lemma}
\begin{proof}
The Peiffer condition yields a commutative diagram
\[
\xymatrix{
M\diamond M \ar[r]^-{i_{M,M}} \ar[d]_-{\partial\diamond 1_M} & M\flat M \ar[r]^-{\chi_M} \ar[d]_-{\partial\flat 1_M} & M \ar@{=}[d]\\
L\diamond M \ar[r]_-{i_{L,M}} & L\flat M \ar[r]_-{\xi} & M.
}
\]
It is clear that the image of the top composition is $[M,M]$, while the image of $\xi\comp i_{L,M}$ is $[L,M]$ because this commutator is the image of
\[
k_\xi\comp \xi\comp i_{L,M}=\binom{i_\xi}{k_\xi}\comp \kappa_{L,M}\comp i_{L,M}=\binom{i_\xi}{k_\xi}\comp \iota_{L,M}\colon L\diamond M\to M\rtimes_\xi L
\]
in $M\leq M\rtimes_\xi L$. Taking cokernels horizontally, we obtain a regular epimorphism ${M}/{\left[M,M\right]}\to{M}/{\left[L,M\right]}$. Hence ${M}/{\left[L,M\right]}$ is abelian, as a quotient of an abelian object.
\end{proof}

\begin{remark}\label{RemarkIsTrivialisation}
This means that for the action $\xi$ of a crossed module $(\partial\colon{M\to L},\xi)$, there is no difference between the procedure of~\ref{subsec:TrivialExtensionReflection} and the construction described in~\ref{ConstructionTrivialAction}. 
\end{remark}

\subsection{Crossed modules over a fixed object}\label{XMod_L}
Let $L$ be an object of $\A$. A \emph{crossed module over $L$} or \emph{$L$-crossed module} is a crossed module $(\partial\colon{M\to L},\xi)$ in $\A$ where the codomain of $\partial$ is the given object $L$. Together with morphisms of crossed modules that keep the object $L$ fixed, this defines the category $\XMod_L(\A)$.

A key issue here is, that $\XMod_L(\A)$ is not a semi-abelian category: indeed, it is not pointed, since the initial $L$-crossed module is $(0\colon 0\to L, \tau^L_0)$, while the terminal $L$-crossed module is $(1_L\colon L\to L, \tau^L_L)$. On the other hand, it is still \emph{quasi-pointed} (in the sense of~\cite{Bou01}, which means that $0\to 1$ is a monomorphism), regular and protomodular, which makes it a so-called \emph{sequentiable} category. Furthermore, it is Barr-exact, so it is actually not far from being semi-abelian. Indeed, results such as Lemma~\ref{lemma:ker of regular pushout is regular epimorphism} stay valid in this context.

\begin{remark}
\label{rmk:kernels in XMOD_L(A)}
Consider a morphism of $L$-crossed modules
\[
(M\xrightarrow{\partial}L,\xi)\xlongrightarrow{(f,1_L)}(M'\xrightarrow{\partial'}L,\xi').
\]
The kernel of this morphism is given by 
\[
(K_f\xrightarrow{0}L,\underline{\xi})\xlongrightarrow{(k_f,1_L)}(M\xrightarrow{\partial}L,\xi),
\]
where the action $\underline{\xi}$ is induced by the universal property of $K_f$ as shown in
\[
\xymatrix{
L\flat K_f \ar@{.>}[d]_-{\underline{\xi}} \ar[r]^-{1_L\flat k_f} & L\flat M \ar[d]_-{\xi} \ar[r]^-{1_L\flat f} & L\flat M' \ar[d]^-{\xi'}\\
K_f \ar[r]_-{k_f} & M \ar[r]_-{f} & M'.
}
\]
It is easy to see that this is an $L$-crossed module.
\end{remark}

\begin{lemma}
\label{lemma:monomorphisms in XMOD_L(A)}
Consider a morphism of $L$-crossed modules
\[
 (M\xrightarrow{\partial}L,\xi)\xlongrightarrow{(f,1_L)}(M'\xrightarrow{\partial'}L,\xi').
\]
Then $(f,1_L)$ is a monomorphism in $\XMod_L(\A)$ if and only if $f$ is mono in $\A$.\noproof 
\end{lemma}

\section{Central extensions of crossed modules, ad-hoc approach}\label{Section Ad Hoc}
We let $\A$ be a semi-abelian category that satisfies \SH. Copying what happens for groups and Lie algebras, we make the following definitions. By~\ref{higgins commutator and central extensions} we know about first one, of course; later on we shall also justify the latter two from a Galois theory perspective.

\begin{definition}\label{DefCentralWRTAb}
A \emph{central extension} in $\A$ (with respect to $\Ab(\A)$) is a regular epimorphism $f\colon {X\to Y}$ with kernel $K$ such that $[K,X]=0$.
\end{definition}

\begin{example}\label{Ex:CentralExtInXMod}
If $(\partial\colon{M\to L},\xi)$ is a crossed module, then the regular epimorphism in the image factorisation of $\partial$ is a central extension; in other words, $[K_\partial,M]=0$.
\end{example}

\begin{definition}\label{Ad-hoc centrality}
Let $L$ be an object of $\A$. A \emph{central extension} in $\XMod_L(\A)$ is a regular epimorphism of $L$-crossed modules
\[
(M\xrightarrow{\partial}L,\xi)\xlongrightarrow{(f,1_L)}(M'\xrightarrow{\partial'}L,\xi')
\]
where for the kernel $K$ of $f$ we have that $[L,K]=0$.
\end{definition}

\begin{remark}\label{Central via kernel has trivial action}
Notice that this means that the kernel $0\colon {K\to L}$ of $(f,1_L)$ has a trivial $L$-action.
\end{remark}

\begin{definition}
Let $L$ be an object of $\A$. A \emph{perfect object} in $\XMod_L(\A)$ is an $L$-crossed module $(\partial\colon{M\to L},\xi)$ such that $[L,M]=M$.
\end{definition}

\begin{lemma}
\label{lemma:lemma 1}
An $L$-crossed module $(\partial\colon{M\to L},\xi)$ is perfect if and only if in the corresponding internal groupoid
\[
\xymatrix{
M\rtimes L\ar@<1ex>[r]^-{d} \ar@<-1ex>[r]_-{c} & L \ar[l]|-{e}
}
\]
the normal closure $\cl_{M\rtimes L}(L)$ of $e$ is all of $M\rtimes L$.
\end{lemma}
\begin{proof}
This follows immediately from the explanation in Remark~\ref{rmk:perfect L-actions}.
\end{proof}

\begin{lemma}[Proposition~3.9 in \cite{EverVdL2}]
\label{lemma:coeq(d,c)=coker(c circ k_d)}
For a reflexive graph with its normalisation
\[
\xymatrixcolsep{3pc}
\xymatrix{
K_d \ar[r]^-{k_d} & C_1 \ar@<1ex>[r]^-{d} \ar@<-1ex>[r]_-{c} & C_0, \ar[l]|-{e}
}
\]
the coequaliser $C_{(d,c)}$ of $d$ and $c$ is isomorphic to the cokernel $C_{c\circ k_d}$ of $c\comp k_d$.\noproof
\end{lemma}

\begin{proposition}
\label{prop:delta_M is a regular epimorphism}
Given a crossed module $(\partial\colon{M\to L},\xi)$ we can construct the crossed square
\begin{equation}\label{DiagTensor}
\vcenter{\xymatrix{
L\otimes M \ar[d] \ar[r]^-{\delta_M} & \ar[d]^-{\partial} M\\
L \ar@{=}[r]& L
}}
\end{equation}
by taking the non-abelian tensor product~\cite{dMVdL19.3}. Then $\delta_M$ is a regular epimorphism if and only if $(\partial\colon{M\to L},\xi)$ is perfect.
\end{proposition}
\begin{proof}
Let us recall the construction of the non-abelian tensor product in this special situation. First of all we denormalise both $(\partial\colon{M\to L},\xi)$ and $(1_L\colon L\to L,\chi_L)$ and we take the pushout of the two split monomorphisms $e$ and $\Delta_L$ to obtain the square of reflexive graphs
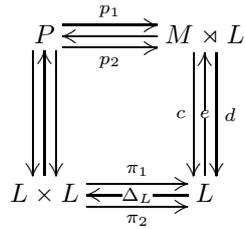
\[
\xymatrix@!0@=3em{
P \ar@<1ex>[dd] \ar@<-1ex>[dd] \ar@<1ex>[rr]^-{p_1} \ar@<-1ex>[rr]_-{p_2} && M\rtimes L \ar[ll] \ar@<1ex>[dd]^-{d} \ar@<-1ex>[dd]_-{c}\\
\\
L\times L \ar[uu] \ar@<1ex>[rr]^-{\pi_1} \ar@<-1ex>[rr]_-{\pi_2} && L \ar[uu]|-{e} \ar[ll]|-{\Delta_L}
}
\]
Then we take a certain quotient of $P$ that universally turns the diagram into a double groupoid as in Figure~\ref{Figure Induced Double Groupoid}. 
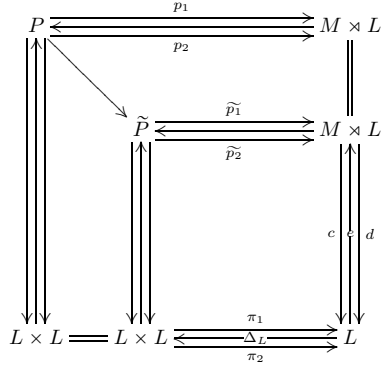
\begin{figure}
\resizebox{.4\textwidth}{!}
{$\xymatrix@!0@=5em{
P \ar[rd] \ar@<1ex>[ddd] \ar@<-1ex>[ddd] \ar@<1ex>[rrr]^-{p_1} \ar@<-1ex>[rrr]_-{p_2} &&& M\rtimes L \ar@{=}[d] \ar[lll]\\
& \widetilde{P} \ar@<1ex>[dd] \ar@<-1ex>[dd] \ar@<1ex>[rr]^-{\widetilde{p_1}} \ar@<-1ex>[rr]_-{\widetilde{p_2}} && M\rtimes L \ar[ll] \ar@<1ex>[dd]^-{d} \ar@<-1ex>[dd]_-{c}\\
\\
L\times L \ar[uuu] \ar@{=}[r] & L\times L \ar[uu] \ar@<1ex>[rr]^-{\pi_1} \ar@<-1ex>[rr]_-{\pi_2} && L \ar[uu]|-{e} \ar[ll]|-{\Delta_L}
}$}
\caption{The induced double groupoid.}\label{Figure Induced Double Groupoid}
\end{figure}
Finally we normalise the entire double groupoid to obtain Figure~\ref{Figure Tensor}. 
\begin{figure}
\resizebox{.4\textwidth}{!}
{$\xymatrix@!0@=7em{
L\otimes M \ar[r] \ar[d] & (L\otimes M)\rtimes M \ar@<1ex>[r] \ar@<-1ex>[r] \ar[d] & M \ar[l] \ar[d]\\
(L\otimes M)\rtimes L \ar@<1ex>[d] \ar@<-1ex>[d] \ar[r]^-{k_{\widetilde{p_1}}} & \widetilde{P} \ar@<1ex>[d] \ar@<-1ex>[d] \ar@<1ex>[r]^-{\widetilde{p_1}} \ar@<-1ex>[r]_-{\widetilde{p_2}} & M\rtimes L \ar[l] \ar@<1ex>[d]^-{d} \ar@<-1ex>[d]_-{c}\\
L \ar[u] \ar[r] & L\times L \ar[u] \ar@<1ex>[r]^-{\pi_1} \ar@<-1ex>[r]_-{\pi_2} & L \ar[u]|-{e} \ar[l]|-{\Delta_L}
}$}
\caption{The tensor product of $(1_L\colon L\to L,\chi_L)$ and $(\partial\colon{M\to L},\xi)$.}\label{Figure Tensor}
\end{figure}
Now we are ready to prove the result. First of all, by applying the ``Short Five Lemma for regular epimorphisms''~\cite[Lemma~4.2.5.5]{BB04} to the diagram 
\begin{align*}
\label{diag:non-abelian tensor product 2x1 version}
\vcenter{
\xymatrixcolsep{3pc}
\xymatrix{
L\otimes M \ar[d] \ar[r]^-{\delta_M} & M \ar[d]\\
(L\otimes M)\rtimes L \ar@<1ex>[d]^-{d_{\otimes}} \ar@<-1ex>[d]_-{c_{\otimes}} \ar[r]^-{\delta_M\rtimes 1_L} & M\rtimes L \ar@<1ex>[d]^-{d} \ar@<-1ex>[d]_-{c}\\
L \ar[u]|-{e_{\otimes}} \ar@{=}[r] & L \ar[u]|-{e}
}
}
\end{align*}
we deduce that $\delta_M$ is a regular epimorphism if and only if $\delta_M\rtimes 1_L$ is so. Then using the fact that all vertical arrows in the diagram
\[
\xymatrixcolsep{4pc}
\xymatrix{
K_{p_1} \ar@/^1pc/[rr]^-{\delta_P} \ar@{.>}[d] \ar[r]_-{k_{p_1}} & P \ar[d] \ar[r]_-{p_2} & M\rtimes L \ar@{=}[d]\\
K_{\widetilde{p_1}} \ar@/_1pc/[rr]_-{\delta_M\rtimes 1_L} \ar[r]^-{k_{\widetilde{p_1}}} & \widetilde{P} \ar[r]^-{\widetilde{p_2}} & M\rtimes L
} 
\]
are regular epimorphisms, we see that $\delta_M\rtimes 1_L$ is a regular epimorphism if and only if $\delta_P$ is so.

Since $\delta_{P}$ is a proper morphism (as a composite of a regular epimorphism with a normal monomorphism in a semi-abelian category), it is a regular epimorphism if and only if it has a trivial cokernel.

On the other hand, Lemma~\ref{lemma:coeq(d,c)=coker(c circ k_d)} says that for every reflexive graph, the cokernel of the normalisation is the same as the coequaliser of the two split epimorphisms: this implies that the first one is trivial if and only if the second one is so. Let us draw a picture involving the desired coequaliser $Q$.
\[
\xymatrix@!0@=3em{
P \ar@<1ex>@{.>}[dd] \ar@<-1ex>@{.>}[dd] \ar@<1ex>[rr]^-{p_1} \ar@<-1ex>[rr]_-{p_2} && M\rtimes L \ar[ll]|-{s} \ar@<1ex>@{.>}[dd]^-{d} \ar@<-1ex>@{.>}[dd]_-{c} \ar[rr]^-{q} && Q \ar@<1ex>@{.>}[dd] \ar@<-1ex>@{.>}[dd]\\
\\
L\times L \ar[uu]|-{e'} \ar@<1ex>[rr]^-{\pi_1} \ar@<-1ex>[rr]_-{\pi_2} && L \ar[uu]|-{e} \ar[ll]|-{\Delta_L} \ar[rr] && 0 \ar[uu]
}
\]
Here the second row involves the coequaliser of $\pi_1$ and $\pi_2$.

Let us prove by hand that $q$ is the cokernel of $e$. Consider $\gamma\colon M\rtimes L\to Z$ such that $\gamma \comp e=0$: since $q$ is the coequaliser of $p_1$ and $p_2$, in order to have a unique morphism $\phi\colon Q\to Z$ such that $\phi\comp q=\gamma$ it suffices that $\gamma\comp p_1=\gamma\comp p_2$. We use the fact that $P$ is a pushout of $e$ and $\Delta_L$, and hence that $(s,e')$ is a jointly epimorphic pair: we have the equalities
\[
\begin{cases}
\gamma\comp p_1\comp s=f=\gamma\comp p_2\comp s\\
\gamma\comp p_1\comp e'=\gamma\comp e\comp\pi_1=0=\gamma\comp e\comp\pi_2=\gamma\comp p_2\comp e'
\end{cases}
\]
and so $\gamma\comp p_1=\gamma\comp p_2$. This means that $Q\cong C_e$.

Finally by Lemma~\ref{lemma:lemma 1} we know that $C_e=0$ if and only if $(\partial\colon{M\to L},\xi)$ is perfect, and this proves our claim.
\end{proof}

\begin{proposition}
\label{prop:any regular epimorphism of L-crossed module is a central extension with respect to Ab}
Consider a regular epimorphism of crossed modules
\[
(M\xrightarrow{\partial}L,\xi)\xlongrightarrow{(f,1_L)}(M'\xrightarrow{\partial'}L,\xi').
\]
Then $f$ considered as a morphism in $\A$ is a central extension (with respect to $\Ab(\A)$, see Definition~\ref{DefCentralWRTAb}).
\end{proposition}
\begin{proof}
Since $(\partial\colon{M\to L},\xi)$ is a crossed module we have that $[K_{\partial},M]=0$. From the commutativity of the triangle
\[
\xymatrix{
M \ar[rr]^-{f} \ar[rd]_-{\partial} && M' \ar[dl]^-{\partial'}\\
& L
}
\]
we can construct a monomorphism $K_{f}\to K_{\partial}$. Since Higgins commutators are monotone, this in turn induces a monomorphism $[K_{f},M]\to [K_{\partial},M]$. Recall (Example~\ref{Ex:CentralExtInXMod}) that $[K_{\partial},M]=0$. Hence also $[K_{f},M]$ is trivial, and therefore $f$ is a central extension (as a morphism in $\A$, with respect to~$\Ab(\A)$---see~\ref{higgins commutator and central extensions}).
\end{proof}

\begin{proposition}
\label{prop:(delta_M,1) is a central extension of L-crossed modules}
If $(\partial\colon{M\to L},\xi)$ is a perfect crossed module, then the morphism $(\delta_M,1_L)$ in \eqref{DiagTensor} is a central extension of $L$-crossed modules.
\end{proposition}
\begin{proof}
We know from the proof of Proposition~\ref{prop:delta_M is a regular epimorphism} that $\delta_M\rtimes 1_L$ in \eqref{DiagTensor} is a regular epimorphism. Since, coming from a crossed square, it is also the differential of a crossed module, it is a central extension in $\A$ with respect to the Birkhoff subcategory $\Ab(\A)$. Via Corollary~\ref{cor:kernel of L-point morphism}, we may picture its kernel in $\Pt_L(\A)$ as the following pullback in $\A$.
\[
\vcenter{\xymatrix@!0@=4em{K_{\delta_M}\rtimes L \dottedpullback \ar@<.5ex>@{.>}[rd]^-{\overline{d}} \ar@{.>}[rr] \ar@{.>}[dd]_-{\overline{d}} && (L\otimes M)\rtimes L \ar@<.5ex>[ld]^-{d_\otimes} \ar[dd]^{\delta_M\rtimes 1_L}\\
& L \ar@<.5ex>[rd]^-{e} \ar@<.5ex>[ru]^-{e_\otimes} \ar@<.5ex>[ld]^-{1_L} \ar@<.5ex>@{.>}[lu]^-{\overline{e}} \\
L \ar[rr]_{e} \ar@<.5ex>[ru]^-{1_L} && M\rtimes L \ar@<.5ex>[lu]^-{d}}}
\]
The morphism $\overline{d}$ is a central extension, as a pullback of $\delta_M\rtimes 1_L$; on the other hand, it is split by $\overline{e}$. Hence by Lemma~\ref{lemma:equivalent conditions to central} it is a trivial extension, which by~\ref{subsec:TrivialExtensionReflection} implies that the action of $L$ on $K_{\delta_M}$ is trivial. This proves our claim.
\end{proof}

\begin{proposition}
\label{prop:central extensions of L-crossed modules induce crossed squares}
Any central extension of $L$-crossed modules 
\[
(M\xrightarrow{\partial}L,\xi)\xlongrightarrow{(f,1_L)}(M'\xrightarrow{\partial'}L,\xi')
\]
induces a crossed square
\[
\xymatrix{
M \ar[r]^-{f} \ar[d]_-{\partial} & M' \ar[d]^-{\partial'}\\
L \ar@{=}[r] & L.
}
\]
\end{proposition}
\begin{proof}
The first step is to prove that $[M\rtimes L,K_{f\rtimes 1_L}]=0$. We use the join decomposition formula of Proposition~\ref{Higgins properties} to see that
\begin{align*}
[M\rtimes L,K_{f\rtimes 1_L}]&=[M,K_{f\rtimes 1_L}]\vee [L,K_{f\rtimes 1_L}]\vee [M,L,K_{f\rtimes 1_L}]
\end{align*}
and we show that each component is trivial:
\begin{itemize}
 \item notice that $K_{f}=K_{f\rtimes 1_L}$ since $f$ is the pullback of $f\rtimes 1_L$;
 \item since $(f,1_L)$ is a central extension, we know that $[L,K_f]=0$;
 \item from Proposition~\ref{prop:any regular epimorphism of L-crossed module is a central extension with respect to Ab} it follows that $f$ is a central extension with respect to $\Ab(\A)$, and therefore $[M,K_f]=0$;
 \item since both $K_f$ and $M$ are normal subobjects of $M\rtimes L$, via~\cite[Section 4]{HVdL11} the \emph{Smith is Huq} condition implies that $[K_f,M,M\rtimes L]\leq [K_f,M]=0$, which in turn implies $[M,L,K_f]=0$ since this is a subobject of the previous one.
\end{itemize} 
Now consider the extension $f\rtimes 1_L$ (it is a regular epimorphism because $f$ is so): since $[M\rtimes L,K_{f\rtimes 1_L}]=0$ we deduce that $f\rtimes 1_L$ is a central extension with respect to $\Ab(\A)$ and therefore it is the differential of a crossed module.

We now use the fact that in $\Grpd(\A)$ the central extensions (with respect to $\Ab(\Grpd(\A))$) are computed pointwise, that is they are couples of central extensions in $\A$ (with respect to $\Ab(\A)$): this is shown in Proposition~$4.1$ of~\cite{BG02}. Since both $f\rtimes 1_L$ and $1_L$ are central with respect to $\Ab(\A)$, the lower square in the diagram
\[
\xymatrixcolsep{3pc}
\xymatrix{
M \ar[d]_-{k_d} \ar[r]^-{f} & M' \ar[d]^-{k_{d'}}\\
M\rtimes L \ar@<1ex>[d]^-{d} \ar@<-1ex>[d]_-{c} \ar[r]^-{f\rtimes 1_L} & M'\rtimes L \ar@<1ex>[d]^-{d'} \ar@<-1ex>[d]_-{c'}\\
L \ar[u]|-{e} \ar@{=}[r] & L \ar[u]|-{e'}
}
\]
is a central extension in $\Grpd(\A)$ (with respect to $\Ab(\Grpd(\A))$) and therefore it is the differential of an internal crossed module in $\Grpd(\A)$. This means that its denormalisation is a double groupoid and therefore the square we are interested in is an internal crossed square.
\end{proof}

We can now use this interpretation of central extensions of $L$-crossed modules in terms of crossed squares in order to prove the following result, of which we show the two implications in separate propositions.

\begin{theorem}\label{Thm:PerfectIFFuce}
\label{thm:perfect iff admits a universal central extension}
In a semi-abelian category that satisfies the \emph{Smith is Huq} condition, an $L$-crossed module is perfect if and only if it admits a universal central extension.
\end{theorem}

\begin{proposition}
Every perfect $L$-crossed module has a universal central extension.
\end{proposition}
\begin{proof}
Let $(\partial\colon{M\to L},\xi)$ be a perfect $L$-crossed module and consider the crossed square
\[
\xymatrix{
L\otimes M \ar[d]_-{\partial^{\otimes}} \ar[r]^-{\delta_M} & \ar[d]^-{\partial} M\\
L \ar@{=}[r]& L.
}
\]
By Proposition~\ref{prop:(delta_M,1) is a central extension of L-crossed modules}, $(\delta_M,1_L)$ is a central extension of $L$-crossed modules. Now we want to show that this central extension is universal---that is, it is initial among all central extensions of $(\partial\colon{M\to L},\xi)$. So consider another central extension
\[
(\overline{M}\xrightarrow{\overline{\partial}}L,\overline{\xi})\xrightarrow{(f,1_L)}(M\xlongrightarrow{\partial}L,\xi).
\]
Due to Proposition~\ref{prop:central extensions of L-crossed modules induce crossed squares} we know that also 
\[
\xymatrix{
\overline{M} \ar[d]_-{\overline{\partial}} \ar[r]^-{f} & \ar[d]^-{\partial} M\\
L \ar@{=}[r]& L
}
\]
is a crossed square. Now it suffices to use the universal property of the non-abelian tensor product (see~\cite{dMVdL19.3}) to conclude that there exists a unique morphism $\phi\colon {L\otimes M\to \overline{M}}$ such that the diagram 
\[
\xymatrix{
L\otimes M \ar@{.>}[rd]|-{\phi} \ar@/_1pc/[rdd]_-{\partial^{\otimes}} \ar@/^1pc/[rrd]^-{\delta_M}\\
& \overline{M} \ar[d]_-{\overline{\partial}} \ar[r]^-{f} & \ar[d]^-{\partial} M\\
& L \ar@{=}[r]& L
}
\]
commutes and the actions are respected. This implies that $(\delta_M,1_L)$ is initial as a central extension of $L$-crossed modules over $(\partial\colon{M\to L},\xi)$.
\end{proof}

\begin{proposition}
In a universal central extension of $L$-crossed modules
\[
(M\xrightarrow{\partial}L,\xi)\xlongrightarrow{(f,1_L)}(M'\xrightarrow{\partial'}L,\xi')
\]
both the domain and the codomain are perfect objects. In particular, any object that admits a universal central extension is perfect.
\end{proposition}
\begin{proof}
Consider an $L$-crossed module $(\partial'\colon {M'\to L},\xi')$ and an abelian object $A$. Since $A$ is abelian, $(0\colon {A\to L},\tau^L_A)$ is an $L$-crossed module. We can construct the crossed module $(\partial'\comp\pi_{M'}\colon A\times M'\to L,\xi_{A\times M'})$
where the action $\xi_{A\times M'}$ is induced by the universal property of the product as shown in the diagram
\begin{equation}
\label{diag:definition of "product action"}
\vcenter{
\xymatrix@R=1em{
& L\flat(A\times M') \ar[dl]_-{1_L\flat\pi_A} \ar[dr]^-{1_L\flat\pi_{M'}} \ar@{.>}[dd]|-{\xi_{A\times M'}}\\
L\flat A \ar[dd]_-{\tau^L_A} && L\flat M' \ar[dd]^-{\xi'}\\
& A\times M' \ar[dl]_-{\pi_A} \ar[dr]^-{\pi_{M'}}\\
A && M'.
}
}
\end{equation}
In order to see that this is an $L$-action it suffices to use the naturality diagrams for $\eta$ and $\mu$ and the fact that both $\tau^L_A$ and $\xi'$ are $L$-actions. Similarly, to see that this gives rise to an $L$-crossed module it suffices to use that both $(\partial'\colon {M'\to L},\xi')$ and $(0\colon {A\to L},\tau^L_A)$ are so.

Now consider the triangle
\[
\xymatrix{
A\times M' \ar[rd]_-{\partial'\comp\pi_{M'}} \ar[rr]^-{\pi_{M'}} && M'. \ar[ld]^-{\partial'}\\
& L
}
\]
This is a morphism of $L$-crossed modules (due to~\eqref{diag:definition of "product action"}) which is a regular epimorphism: we want to follow Remark~\ref{Central via kernel has trivial action} and show that it is a central extension by proving that its kernel has a trivial $L$-action. But its kernel is simply $(0\colon {A\to L},\tau^L_A)$: to see this is suffices to use the description of kernels in $\XMod_L(\A)$, to notice that $A=K_{\pi_{M'}}$ in the base category $\A$ and to show the commutativity of the square on the left in the diagram
\[
\xymatrixcolsep{3pc}
\xymatrix{
L\flat A \ar[d]_-{\tau^L_A} \ar[r]^-{1_L\flat \langle 1_A,0\rangle} & L\flat(A\times M') \ar[d]_-{\xi_{A\times M'}} \ar[r]^-{1_L\flat \pi_{M'}} & L\flat M' \ar[d]_-{\xi'}\\
A \ar[r]_-{\langle 1_A,0\rangle} & A\times M' \ar[r]_-{\pi_{M'}} & M'.
}
\]
We conclude that, since its action is trivial, $(\pi_{M'},1_L)$ is a central extension.

Now suppose that 
\[
(M\xrightarrow{\partial}L,\xi)\xlongrightarrow{(f,1_L)}(M'\xrightarrow{\partial'}L,\xi')
\]
is a universal central extension of $L$-crossed modules. By definition, we have a unique morphism
\[
\xymatrix@!0@R=4em@C=8em{
(M\xrightarrow{\partial}L,\xi) \ar@{.>}[rr]^-{(\langle g,f\rangle,1_L)} \ar[rd]_-{(f,1_L)} && (A\times M'\xrightarrow{\partial'\circ\pi_{M'}}L,\xi_{A\times M'}) \ar[ld]^-{(\pi_{M'},1_L)}\\
& (M'\xrightarrow{\partial'}L,\xi') 
}
\]
from this extension to the one just defined. Let us focus on this induced morphism: what can we say about $g\colon M\to A$? It is the unique morphism that makes $(\langle g,f\rangle,1_L)$ a morphism of $L$-crossed modules, that is such that the following squares commute.
\begin{align*}
\xymatrixcolsep{3pc}
\xymatrix{
M \ar[d]_-{\langle g,f\rangle} \ar[r]^-{\partial} & L \ar@{=}[d]\\
A\times M' \ar[r]_-{\partial'\comp\pi_{M'}} & L
}
&&
\xymatrixcolsep{3pc}
\xymatrix{
L\flat M \ar[d]_-{\xi} \ar[r]^-{1_L\flat \langle g,f\rangle} & L\flat (A\times M') \ar[d]^-{\xi_{A\times M'}}\\
M \ar[r]_-{\langle g,f\rangle} & A\times M'
}
\end{align*}
The first one does so for each choice of $g$, whereas the second one will if and only if
\begin{equation}
\label{diag:g is a morphism of L-actions}
\vcenter{
\xymatrix{
L\flat M \ar[d]_-{\xi} \ar[r]^-{1_L\flat g} & L\flat A \ar[d]^-{\tau^L_A}\\
M \ar[r]_-{g} & A
}
}
\end{equation}
commutes. Now, since 
\[
\xymatrix{
0 \ar[r] & L\diamond M \ar[r]^-{i_{L,M}} & L\flat M \ar@<.5ex>[r]^-{\tau^L_M} & M \ar@<.5ex>[l]^-{\eta^L_M} \ar[r] & 0
}
\]
is a split short exact sequence, we have that $\binom{\eta^L_M}{\tau^L_M}\colon (L\diamond M)+M\to L\flat M$ is an epimorphism. Hence the commutativity of~\eqref{diag:g is a morphism of L-actions} is equivalent to the commutativity of the same diagram composed with this epimorphism. This amounts to having that $g\comp \xi\comp i_{L,M}=0$, which is another way to say that $g([L,M])=0$. The morphism~$g$ is unique in $\Hom(M,A)$ with this property.

Now fix $A={M}/{[L,M]}$, which is an abelian object by Lemma~\ref{lemma:M over the commutator is abelian}. We are going to deduce that $[L,M]=M$. Notice that both the quotient $g=q\colon M\twoheadrightarrow {M}/{[L,M]}$ and the zero morphism $g=0$ satisfy the condition $g([L,M])=0$. We may thus conclude that $q=0$, so that $[L,M]=M$. This means that $(\partial\colon{M\to L},\xi)$ is perfect and consequently $(\partial'\colon {M'\to L},\xi')$ is perfect too, as a quotient of a perfect object.
\end{proof}

\begin{proof}[Proof of Theorem~\ref{Thm:PerfectIFFuce}]
Combine the two previous propositions.
\end{proof}

\section{Galois theory interpretation, quasi-pointed setting}\label{Section Quasipointed}

The aim here is to use the coinvariants reflector to construct a Birkhoff subcategory of $\XMod_L(\A)$ with respect to which we find the ``right'' class of central extensions of $L$-crossed modules. In~\cite{Eda19}, the author solved this problem in the case of $\A$ being the category of Lie algebras.

In the current section, we work in $\XMod_L(\A)$, which forces us to take into account the lack of a zero object---see \ref{XMod_L}. Here and there we may simplify the situation by making constructions in the more benign context of $\XMod(\A)$. In the next section, we abandon $\XMod_L(\A)$ altogether in favour of $\XMod(\A)$; as we shall see, this simplifies the situation a lot.

\begin{definition}\label{Def:AAXMod}
Given an internal crossed module $(\partial\colon{M\to L},\xi)$, we write $K_{\partial}$ for the kernel of $\partial$, and let $[L,M]\lhd M\rtimes_\xi L$ be the commutator induced by $\xi$ as in Proposition~\ref{prop:higgins=coinvariance}.

We say that $(\partial\colon{M\to L},\xi)$ is \emph{action-acyclic} when $K_{\partial}\wedge [L,M]=0$. Here the intersection is the subobject of $M$ defined via the pullback
\[
\xymatrix{
K_{\partial}\wedge [L,M] \pullback \ar@{.>}[dr]|-{i} \ar[r] \ar[d] & [L,M] \ar[d]^-{k_{(c_s\circ k_p)}}\\
K_{\partial} \ar[r]_-{k_{\partial}} & M.
}
\]
The idea behind this definition is that the action has no cycles (elements of $K_{\partial}$) in its image. 

We will denote $\TXMod_L(\A)$ the full subcategory of $\XMod_L(\A)$ whose objects are the action-acyclic crossed modules.
\end{definition}

Notice that since $i$ is the diagonal of the pullback of a kernel along another kernel, it is itself a kernel. Furthermore, the intersection $K_{\partial}\wedge [L,M]$ is abelian, because $K_\partial$ is the kernel of a central extension as in Example~\ref{Ex:CentralExtInXMod}. This allows us to use the following lemma:

\begin{lemma}\label{Lem:AbelianObject}
$A\in\Ab(\A)$ if and only if $(0\colon{A\to 0},\tau^0_A)$ is an internal crossed module.
\end{lemma}
\begin{proof}
An internal groupoid structure on the reflexive graph $\xymatrix@1{A \ar@<1ex>[r] \ar@<-1ex>[r] & 0 \ar[l]}$ is the same thing as internal monoid structure on $A$, which in the current context amounts to an internal abelian group structure.
\end{proof}

\begin{construction}
\label{constr: free action-acyclic L-crossed module functor}
We define the functor $F\colon\XMod_L(\A)\to\TXMod_L(\A)$, left adjoint to the inclusion functor $J$. 

Given an internal $L$-crossed module $(\partial\colon{M\to L},\xi)$, the sub--crossed module $(0\colon (K_{\partial}\wedge[L,M])\to 0,\tau^0_{K_{\partial}\wedge[L,M]})$ is obtained via Lemma~\ref{Lem:AbelianObject}. The inclusion between the two crossed modules is given by the morphism $(i,0)\in\XMod(\A)$:
\begin{align*}
\vcenter{\xymatrix{
0\flat(K_{\partial}\wedge[L,M]) \ar[d]_-{\tau^0_{K_{\partial}\wedge[L,M]}} \ar[r]^-{0\flat i} & L\flat M \ar[d]^-{\xi}\\
(K_{\partial}\wedge[L,M]) \ar[r]_-{i} & M
}}
&&
\vcenter{\xymatrix{
(K_{\partial}\wedge[L,M]) \ar[r]^-{i} \ar[d]_-{0} & M \ar[d]^-{\partial}\\
0 \ar[r]_-{0} & L.
}}
\end{align*}

The image of $(\partial\colon{M\to L},\xi)$ through $F$ is given by the cokernel in $\XMod(\A)$ of the previous inclusion, that is
\[
(\tfrac{M}{K\wedge [L,M]}\xlongrightarrow{\overline{\partial}}L,\overline{\xi})
\]
where the action $\overline{\xi}$ is obtained as follows: first we pass to the category of points and take the cokernel there
\begin{equation}
\label{diag:cokernel of points}
\vcenter{
\xymatrix{
0 \ar[r] & K_{\partial}\wedge[L,M] \ar[d]_-{i} \ar@{=}[r] & K_{\partial}\wedge[L,M] \ar[d]_-{k_p\circ i} \ar@<.5ex>[r]^-{0} & 0 \ar@<.5ex>[l]^-{0} \ar[r] \ar[d]^-{0} & 0\\
0 \ar[r] & M \ar@{.>}[d] \ar[r]^-{k_p} & X \ar[d]|-{c_{(k_p\circ i)}} \ar@<.5ex>[r]^-{p} & L \ar@{=}[d] \ar@<.5ex>[l]^-{s} \ar[r] & 0\\
0 \ar[r] & K_{\overline{p}} \ar[r]_-{k_{\overline{p}}} & C_{(k_p\circ i)} \ar@<.5ex>[r]^-{\overline{p}} & L \ar@<.5ex>[l]^-{\overline{s}} \ar[r] & 0,
}
}
\end{equation}
then we go back to the associated action $\overline{\xi}$ given by the diagram
\[
\xymatrixcolsep{4pc}
\xymatrix{
L\flat K_{\overline{p}} \ar@{.>}[d]_-{\overline{\xi}} \ar[r]^-{k_{L,k_{\overline{p}}}} & L+K_{\overline{p}} \ar[d]^-{\binom{\overline{s}}{k_{\overline{p}}}} \ar[r]^-{\binom{1_L}{0}} & L \ar@{=}[d]\\
K_{\overline{p}} \ar[r]_-{k_{\overline{p}}} & C_{(k_p\circ i)} \ar[r]_-{\overline{p}} & L.
}
\]
The first thing we need, is to prove that $K_{\overline{p}}\cong C_i$; this follows easily from the fact that the dotted arrow in \eqref{diag:cokernel of points} is a regular epimorphism whose kernel is $i$, all because the lower left square in \eqref{diag:cokernel of points} is a pullback.

At this point one would expect that the action $\overline{\xi}$ just defined makes the diagram 
\[
\xymatrixcolsep{4pc}
\xymatrix{
0\flat(K_{\partial}\wedge[L,M]) \ar[d]_-{\tau^0_{K_{\partial}\wedge[L,M]}} \ar[r]^-{0\flat i} & L\flat M \ar@{}[rd]|-{(*)} \ar[d]^-{\xi} \ar[r]^-{1_L\flat c_i} & L\flat\frac{M}{K_{\partial}\wedge [L,M]} \ar[d]^-{\overline{\xi}}\\
(K_{\partial}\wedge[L,M]) \ar[r]_-{i} & M \ar[r]_-{c_i} & \frac{M}{K_{\partial}\wedge [L,M]}
}
\]
commute and indeed this can be shown by using the diagram in Figure~\ref{OvXi is action}. 
\begin{figure}
\[
\xymatrix@!0@=4.5em{
L\flat M
\ar[dd]_-{\xi}
\ar[rr]
\ar[rd]|-{1_L\flat c_i}
&&
L+M
\ar[dd]|(.50){\hole}_(0.25){\binom{s}{k_p}}
\ar[rr]^-{\binom{1_L}{0}}
\ar[rd]|-{1_L+c_i}
&&
L
\ar@{=}[dd]|(.50){\hole}
\ar@{=}[rd]
\\
&
L\flat \frac{M}{K_{\partial}\wedge [L,M]}
\ar[dd]_(0.33){\overline{\xi}}
\ar[rr]
&&
L+\frac{M}{K_{\partial}\wedge [L,M]}
\ar[dd]_(0.25){\binom{\overline{s}}{k_{\overline{p}}}}
\ar[rr]^-{\binom{1_L}{0}}
&&
L
\ar@{=}[dd]
\\
M
\ar[rr]|(.50){\hole}_(0.33){k_p}
\ar[rd]|-{c_i}
&&
X
\ar[rd]|-{c_{(k_p\circ i)}}
\ar[rr]|(.50){\hole}_(0.66){p}
&&
L
\ar@{=}[rd]
\\
&
\frac{M}{K_{\partial}\wedge [L,M]}
\ar[rr]_-{k_{\overline{p}}}
&&
C_{(k_p\circ i)}
\ar[rr]_-{\overline{p}}
&&
L
}
\]
\caption{$\overline{\xi}$ is an action.}\label{OvXi is action}
\end{figure}
The morphism $\overline{\partial}$ is induced via the universal property of the cokernel of $i$:
\[
\xymatrixcolsep{3pc}
\xymatrix{
(K_{\partial}\wedge[L,M]) \ar[r]^-{i} \ar[d]_-{0} & M \ar@{}[rd]|-{(**)} \ar[d]^-{\partial} \ar[r]^-{c_i}& \frac{M}{K_{\partial}\wedge [L,M]} \ar@{.>}[d]^-{\overline{\partial}}\\
0 \ar[r]_-{0} & L \ar@{=}[r] & L.
}
\]
The fact that it is an internal crossed module is easy to show: it suffices to use that $(\partial\colon{M\to L},\xi)$ is an internal crossed module and that both $q\flat q$ and $1_L\flat q$ are (regular) epimorphisms (by Lemma~5.11 in \cite{MM10}). From the commutativity of $(*)$ and $(**)$ we conclude that 
\[
\xymatrix{
(M\xrightarrow{\partial}L,\xi) \ar[rr]^-{(c_i,1_L)} && (\frac{M}{K_{\partial}\wedge [L,M]}\xlongrightarrow{\overline{\partial}}L,\overline{\xi})
}
\]
is a morphism of $L$-crossed modules. Furthermore it can easily be checked that this morphism has the universal property of the cokernel of $(i,0)$ in $\XMod(\A)$.
\end{construction}

\begin{lemma}
The category $\TXMod_L(\A)$ is Birkhoff in $\XMod_L(\A)$. 
\end{lemma}
\begin{proof}
This proof follows the pattern of Proposition~\ref{prop:birkhoff subcategory of actions}: by Corollary~5.7 in~\cite{EverVdL1}, it suffices that the subfunctor of $1_{\XMod(\A)}$ determined by the construction in Definition~\ref{Def:AAXMod} preserves regular epimorphisms. So, consider a morphism as in 
\begin{equation*}
\label{diag:morphism between L-crossed modules}
(M\xrightarrow{\partial}L,\xi)\xlongrightarrow{(f,1_L)}(M'\xrightarrow{\partial'}L,\xi')
\end{equation*}
which is also a regular epimorphism in $\XMod_L(\A)$. Due to Lemma~\ref{lemma:regular epimorphisms in XMOD(A)}, this means that $f$ is a regular epimorphism in $\A$. Consider the cube
\[
\resizebox{.4\textwidth}{!}
{$
\label{diag:cube induced by morphism}
\vcenter{\xymatrix@!0@=5em{
K_{\partial}\wedge [L,M]
\ar[rr]
\ar[dd]
\ar[rd]|-{\phi}
&&
[L,M]
\ar[dd]|-{\hole}
\ar[rd]|-{[1_L,f]}
\\
&
K_{\partial'}\wedge [L,M']
\ar[rr]
\ar[dd]
&&
[L,M']
\ar[dd]
\\
K_{\partial}
\ar[rr]^(.33){k_{\partial}}|(0.5){\hole}
\ar[rd]
&&
M
\ar[rd]|-{f}
\\
&
K_{\partial'}
\ar[rr]_-{k_{\partial'}}
&&
M'.
}}$}
\]
Its front and back faces are pullbacks by definition of the intersection, while the bottom face is a pullback since $(f,1_L)$ is a morphism of crossed modules over $L$. Hence the top square is a pullback as well. It follows that $\phi$ is a regular epimorphism, because so is $[1_L,f]$.
\end{proof}

A functor is \emph{protoadditive} when it preserves split short exact sequences. This concept was introduced in~\cite{EG-honfg} and further studied in~\cite{Everaert-Gran-TT}, in the context of pointed protomodular categories. Here we just need to explain that in a quasi-pointed category, kernels and cokernels are defined by pulling back and pushing out along the zero object: see~\cite{Bou01}.

\begin{theorem}
\label{thm:F is protoadditive}
The reflector $F$ is protoadditive.
\end{theorem}
\begin{proof}
The proof is made of the following steps:
\begin{enumerate}
 \item Show that the functor that sends an $L$-crossed module $(\partial\colon{M\to L},\xi)$ to the commutator $[L,M]$ is protoadditive;
 \item show that the functor $(\partial\colon{M\to L},\xi)\mapsto (K_{\partial}\wedge[L,M])$ is protoadditive;
 \item use the $3\times 3$-Lemma to conclude that $F$ is protoadditive.
\end{enumerate}
For what regards (1) the aim is to prove that any split short exact sequence of $L$-crossed modules 
\begin{equation}
\label{diag:split short exact sequence of L-crossed module}
\xymatrix{
(K\xrightarrow{0}L,\underline{\xi}) \ar[r]^-{(k,1_L)} & (M\xrightarrow{\partial}L,\xi) \ar@<.5ex>[r]^- {(f,1_L)} & (M'\xrightarrow{\partial'}L,\xi') \ar@<.5ex>[l]^-{(g,1_L)}
}
\end{equation}
induces a split short exact sequence of Higgins commutators 
\[
\xymatrix{
0 \ar[r] & [L,K] \ar[r]^-{[1_L,k]} & [L,M] \ar@<.5ex>[r]^-{[1_L,f]} & [L,M'] \ar@<.5ex>[l]^-{[1_L,g]} \ar[r] & 0
}
\]
From the fact that 
\[
\xymatrix{
0 \ar[r] & K \ar[r]^-{k} & M \ar@<.5ex>[r]^-{f} & M' \ar@<.5ex>[l]^-{g} \ar[r] & 0
}
\]
is a split exact sequence in the base category, by using Proposition~$2.24$ in~\cite{HVdL11} we obtain that 
\[
\xymatrix{
0 \ar[r] & (L\diamond K\diamond M)\rtimes (L\diamond K) \ar[r] & L\diamond M \ar@<.5ex>[r]^-{1_L\diamond f} & L\diamond M' \ar@<.5ex>[l]^-{1_L\diamond g} \ar[r] & 0
}
\]
is a split exact sequence as well. We have the comparison arrows
\[
\xymatrix{
0 \ar[r] & (L\diamond K\diamond M)\rtimes (L\diamond K) \ar[d] \ar[r] & L\diamond M \ar[d]_-{\xi^{\diamond}} \ar@<.5ex>[r]^-{1_L\diamond f} & L\diamond M' \ar[d]^-{\xi'^{\diamond}} \ar@<.5ex>[l]^-{1_L\diamond g} \ar[r] & 0\\
0 \ar[r] & K \ar[r]_-{k} & M \ar@<.5ex>[r]^-{f} & M' \ar@<.5ex>[l]^-{g} \ar[r] & 0
}
\]
whose images, from right to left, are $[L,M']$, $[L,M]$ and $[L,K,M]\vee [L,K]$. These images form a split short exact sequence: taking kernels to the left, 
\[
\xymatrix{0 \ar[r] & (L\diamond K\diamond M)\rtimes (L\diamond K) \ar[d] \ar[r] & L\diamond M \ar[d] \ar@<.5ex>[r]^-{1_L\diamond f} & L\diamond M' \ar[d] \ar@<.5ex>[l]^-{1_L\diamond g} \ar[r] & 0\\
0 \ar[r] & K_{[1_L,f]} \ar[d] \ar[r] & [L, M] \ar[d] \ar@<.5ex>[r]^-{[1_L,f]} & [L,M'] \ar[d] \ar@<.5ex>[l]^-{[1_L,g]} \ar[r] & 0\\
0 \ar[r] & K \ar[r]_-{k} & M \ar@<.5ex>[r]^-{f} & M' \ar@<.5ex>[l]^-{g} \ar[r] & 0}
\]
we see that the bottom left square is a pullback, since the bottom right vertical arrow is a monomorphism; likewise, the top right square is a regular pushout, so the top left vertical arrow is a regular epimorphism. It follows that $K_{[1_L,f]}$ is the image $[L,K,M]\vee [L,K]$ of the left vertical composite. Since $K$ is central in $M$---here we use the underlying crossed module structure---we have $[L,K,M]\leq [M,K,M]\leq [K,M]=0$, so that $K_{[1_L,f]}=[L,K]$.

For (2), consider the diagram in Figure~\ref{Fig:SASCube}. 
\begin{figure}
\resizebox{.5\textwidth}{!}
{$\xymatrix@!0@=5.5em{
K_0\wedge [L,K]
\ar@{=}[rrr]
\ar[rd]^-{\priem{k}}
\ar[ddd]
&&&
[L,K]
\ar[ddd]|(0.33){\hole}|(0.66){\hole}
\ar[rd]^-{[1_L,k]}
\\
&
K_{\partial}\wedge [L,M]
\ar[rrr]^(0.5){\underline{k_{\partial}}}
\ar@<.5ex>[rd]^-{\priem{f}}
\ar[ddd]
&&&
[L,M]
\ar@<.5ex>[rd]^-{[1_L,f]}
\ar[ddd]|(0.33){\hole}
\\
&&
K_{\partial'}\wedge [L,M']
\ar[rrr]^(0.5){\underline{k_{\partial'}}}
\ar[ddd]
\ar@<.5ex>[lu]^-{\priem{g}}
&&&
[L,M']
\ar@<.5ex>[lu]^-{[1_L,g]}
\ar[ddd]
\\
K_0
\ar@{=}[rrr]|(0.33){\hole}|(0.66){\hole}
\ar[rd]
&&&
K
\ar[rd]^-{k}
\\
&
K_{\partial}
\ar[rrr]|(0.33){\hole}^(0.5){k_{\partial}}
\ar@<.5ex>[rd]
&&&
M
\ar@<.5ex>[rd]^-{f}
\\
&&
K_{\partial'}
\ar[rrr]^(0.5){k_{\partial'}}
\ar@<.5ex>[lu]
&&&
M'
\ar@<.5ex>[lu]^-{g}
}$}
\caption{A square of split exact sequences from the proof of Theorem~\ref{thm:F is protoadditive}.}\label{Fig:SASCube}
\end{figure}
It is trivial that $[1_L,f]\comp[1_L,g]=1_{[L,M']}$. Then it remains to show that $\priem{k}=k_{\priem{f}}$. Suppose that $\alpha\colon {A\to K_{\partial}\wedge[L,M]}$ satisfies $\priem{f}\comp \alpha=0$. Then $0=\underline{k_{\partial'}}\comp\priem{f}\comp \alpha=[1_L,f]\comp\underline{k_{\partial}}\comp \alpha$. Since $[1_L,k]=k_{[1_L,f]}$, we have a unique $\gamma\colon A\to [L,K]$ such that $[1_L,k]\comp\gamma=\underline{k_{\partial}}\comp \alpha$. Using that $\underline{k_{\partial}}$ is a monomorphism, from the equality $[1_L,k]=\underline{k_{\partial}}\comp \priem{k}$ we deduce that $\priem{k}\comp\gamma=\alpha$.

Finally, in order to prove (3), consider the diagram in Figure~\ref{Fig:3x3}. 
\begin{figure}
\resizebox{.7\textwidth}{!}
{$\xymatrix{
(K_{0}\wedge[L,K]\xrightarrow{0}L,\underline{\phi}) \ar[d] \ar[r]^-{(\priem{k},1_L)} & (K_{\partial}\wedge[L,M]\xrightarrow{0}L,\phi) \ar[d] \ar@<.5ex>[r]^-{(\priem{f},1_L)} & (K_{\partial'}\wedge[L,M']\xrightarrow{0}L,\phi') \ar[d] \ar@<.5ex>[l]^-{(\priem{g},1_L)}\\
(K\xrightarrow{0}L,\underline{\xi}) \ar[d] \ar[r]^-{(k,1_L)} & (M\xrightarrow{\partial}L,\xi) \ar[d] \ar@<.5ex>[r]^-{(f,1_L)} &
(M'\xrightarrow{\partial'}L,\xi') \ar[d] \ar@<.5ex>[l]^-{(g,1_L)}\\
F(K\xrightarrow{0}L,\underline{\xi}) \ar[r]^-{F(k,1_L)} & F(M\xrightarrow{\partial}L,\xi) \ar@<.5ex>[r]^-{F(f,1_L)} &
F(M'\xrightarrow{\partial'}L,\xi') \ar@<.5ex>[l]^-{F(g,1_L)}
}$}
\caption{$3\times 3$-diagram in the proof of Theorem~\ref{thm:F is protoadditive}.}\label{Fig:3x3}
\end{figure}
Each column is exact by definition of the functor $F$ and the middle row is exact by hypothesis. From the description of kernels in Remark~\ref{rmk:kernels in XMOD_L(A)} and from the previous step, we deduce that the top row is exact as well. Now it suffices to use the $3\times 3$-Lemma to obtain that also the bottom row is exact.
\end{proof}

\begin{remark}
When $\A$ is a strongly protomodular category (see~\cite{BB04}) we can give a simpler proof of the protoadditivity of the functor $F$ by changing the way in which (1) is shown in Theorem~\ref{thm:F is protoadditive}. This proof uses Proposition~\ref{prop:any regular epimorphism of L-crossed module is a central extension with respect to Ab} as follows.

Consider a split short exact sequence of $L$-crossed modules as in~\eqref{diag:split short exact sequence of L-crossed module}: by Proposition~\ref{prop:any regular epimorphism of L-crossed module is a central extension with respect to Ab} we know that $f$ is a central extension in $\A$ (with respect to $\Ab(\A)$), but since it is also split, it is a trivial extension and hence a product projection. In particular, $g$ is a normal monomorphism, and since $\A$ is a strongly protomodular category, it follows that $(g,1_L)$ is a normal monomorphism of $L$-actions: since it is a split monomorphism as well, we see that $(g,1_L)$ is a product inclusion and $(f,1_L)$ is a product projection in $\Act_L(\A)$.

Notice that in the semi-abelian context, regular epi--reflectors preserve products. In particular, $F$ preserves products, so it sends the split short exact sequence~\eqref{diag:split short exact sequence of L-crossed module} into a sequence which is again split exact. Finally by the $3\times 3$-Lemma we deduce that the induced sequence of commutators is split short exact as well.
\end{remark}

Now, using Lemma~\ref{lemma:equivalent conditions to central} we can reformulate centrality as follows.

\begin{lemma}\label{Central via good kernel}
An extension $(f,1_L)\colon (\partial\colon M\to L,\xi)\to(\partial'\colon M'\to L,\xi')$ is central with respect to $\TXMod_L(\A)$ if and only if its kernel 
\[
\xymatrix{
K_{(f,1_L)}=(K_f\xrightarrow{0}L,\underline{\xi}) \pullback \ar[r] \ar[d] & (M\xrightarrow{\partial}L,\xi) \ar[d]^-{(f,1_L)}\\
(0\xrightarrow{0}L,\tau^L_0) \ar[r] & (M'\xrightarrow{\partial'}L,\xi')
}
\]
is an action-acyclic crossed module.
\end{lemma}
\begin{proof}
The characterisation of central extensions as those extensions whose kernel lies in the given Birkhoff subcategory is actually valid for protoadditive Birkhoff reflectors in general~\cite{Everaert-Gran-TT}, at least in the homological context; we repeat the argument for the sake of convenience. 

We write $((\partial''\colon M\times_{M'}M\to L,\xi''),(\pi_1,1_L),(\pi_2,1_L))$ for the kernel pair of $(f,1_L)$. The given extension is then central if and only if the projection
\[
(\pi_1,1_L)\colon (M\times_{M'}M\xrightarrow{\partial''} L,\xi'')\to (M\xrightarrow{\partial}L,\xi) 
\]
is a trivial extension. Since $\pi_1$ is a split epimorphism, we have the diagram
\[
\label{diag:morphism of short exac sequence}
\vcenter{
\xymatrixcolsep{1.5pc}
\xymatrix{
0 \ar[r] & K_{(f,1_L)} \ar[r] \ar[d] & (M\times_{M'}M\xrightarrow{\partial''} L,\xi'') \ar[d] \ar@<.5ex>[rr]^-{(\pi_1,1_L)} && (M\xrightarrow{\partial}L,\xi) \ar[d] \ar@<.5ex>[ll] \ar[r] & 0\\
0 \ar[r] & F(K_{(f,1_L)}) \ar[r] & F((M\times_{M'}M\xrightarrow{\partial''}L,\xi'')) \ar@<.5ex>[rr]^-{F((\pi_1,1_L))} && F((M\xrightarrow{\partial}L,\xi)) \ar@<.5ex>[ll] \ar[r] & 0
}
}
\]
where the vertical morphisms are the components of the unit. Notice that the first row is exact since $K_{(\pi_1,1_L)}=K_{(f,1_L)}$. Hence the second row is exact, because $F$ is protoadditive. By definition we have that $(\pi_1,1_L)$ is a trivial extension if and only if the square on the right is a pullback, but this is true if and only if the vertical morphism on the left is an isomorphism~\cite[Proposition~$7$]{Bou01}. 
\end{proof}

\begin{theorem}
\label{cor:equivalent condition to centrality}
An extension of $L$-crossed modules in $\A$ is central with respect to $\TXMod_L(\A)$ if and only if it is a central extension in the sense of Definition~\ref{Ad-hoc centrality}.
\end{theorem}
\begin{proof}
This is a consequence of Lemma~\ref{Central via good kernel} and Proposition~\ref{prop:higgins=coinvariance}.
\end{proof}

We have to generalise Definition~\ref{defi:perfect object in pointed setting} to the quasi-pointed \cite{Bou01} exact environment of $\XMod_L(\A)$. There seems to be no single categorically sound approach to this; so we stick with the following ad-hoc interpretation:

\begin{definition}
\label{defi:perfect object in XMod_L(A) external}
Given an $L$-crossed module $(\partial\colon{M\to L},\xi)$, we say that it is \emph{perfect} (with respect to $\TXMod_L(\A)$) whenever its underlying action is perfect (with respect to $\TrivAct_L(\A)$), which means that $[L,M]=M$.
\end{definition}

The aim of the next section is to make this more natural: we set up a Galois theory with respect to which both the central extensions and the perfect objects agree with those needed in Section~\ref{Section Ad Hoc}.

\section{Galois theory interpretation, pointed setting}\label{Section Pointed}

Given an internal crossed module $(\partial\colon{M\to L},\xi)$, Lemma~\ref{lemma:M over the commutator is abelian} tells us that the quotient ${M}/{[L,M]}$ is always an abelian object. Clearly, this induces a functor determined by
\begin{align*}
 F\colon\XMod(\A)\to \Ab(\A)\colon (M\xrightarrow{\partial}L,\xi)\mapsto \frac{M}{[L,M]},
\end{align*}
which has a right adjoint given by the inclusion of abelian objects as particular crossed modules. Indeed, via Lemma~\ref{Lem:AbelianObject} we obtain a functor
\begin{align*}
 G\colon\Ab(\A)\to \XMod(\A)\colon A\mapsto (A\xrightarrow{0}0,\tau^0_A).
\end{align*}
which allows us to view $\Ab(\A)$ as a subcategory of $\XMod(\A)$. As before, we may follow the pattern of Proposition~\ref{prop:birkhoff subcategory of actions} to show that $F$ is a Birkhoff reflector with right adjoint $G$; this follows immediately from Corollary~5.7 in~\cite{EverVdL1} and the fact that the commutator functor $[L,-]$ preserves regular epimorphisms. Thus we see:

\begin{proposition}
The category $\Ab(\A)$ is a Birkhoff subcategory of $\XMod(\A)$, with reflector $F$ whose right adjoint is $G$.\noproof
\end{proposition}

\begin{remark}
If $\XMod(\A)$ has enough projectives then so does $\A$, since $\A$ is included as a Birkhoff subcategory and Birkhoff reflectors preserve the property of existence of enough projectives. (Indeed, any left adjoint whose right adjoint preserves regular epimorphisms does so.) The converse (that $\XMod(\A)$ has enough projectives if so does $\A$) is a special case of Proposition 3.2 in \cite{EG-honfg}.
\end{remark}

Now we are able to apply Theorem~$3.5$ in~\cite{CVdL14} to obtain the following.

\begin{corollary}
\label{cor:theorem 3.5 in our particular case}
Suppose $\A$ is semi-abelian with~\SH\ and enough projectives. An internal crossed module $(\partial\colon{M\to L},\xi)$ of $\A$ is perfect (with respect to the Birkhoff subcategory $\Ab(\A)$ of $\XMod(\A)$) if and only if it admits a universal central extension (with respect to the Birkhoff subcategory $\Ab(\A)$ of $\XMod(\A)$).\noproof
\end{corollary}

We still have to explain why the central extensions and the perfect objects in this sense agree with the definitions above. Once this is clear, we find Theorem~\ref{thm:perfect iff admits a universal central extension} as a consequence of Corollary~\ref{cor:theorem 3.5 in our particular case}---under the condition that enough projectives exist in $\A$. If $\A$ happens to lack projectives, then Theorem~\ref{thm:perfect iff admits a universal central extension} stays valid, of course.

\begin{proposition}
\label{prop:equivalence between two concepts of universal central extension}
Given an extension of a crossed module $(\partial\colon{M\to L},\xi)$, it is a (universal) central extension with respect to the Birkhoff subcategory 
\begin{equation}
\label{diag:first Birkhoff subcategory}
\xymatrix{
\Ab(\A) \ar@<.5ex>[r] & \XMod(\A), \ar@<.5ex>[l]
}
\end{equation}
if and only if it is a (universal) central extension with respect to the Birkhoff subcategory 
\begin{equation}
\label{diag:second Birkhoff subcategory}
\xymatrix{
\TXMod_L(\A) \ar@<.5ex>[r] & \XMod_L(\A). \ar@<.5ex>[l]
}
\end{equation}
\end{proposition}

Before we prove this, we need a lemma:

\begin{lemma}
\label{lemma:equivalent condition to triviality}
Consider an extension of crossed modules 
\begin{equation}
\label{eq:extension of crossed modules}
(M'\xrightarrow{\partial'}L',\xi')\xrightarrow{(f,l)}(M\xrightarrow{\partial}L,\xi)
\end{equation}
which is central with respect to~\eqref{diag:first Birkhoff subcategory}. Then $l$ is an isomorphism and $(f,l)$ can be considered as an extension of $L$-crossed modules.
\end{lemma}
\begin{proof}
Let us start by proving that a morphism as in~\eqref{eq:extension of crossed modules} is a trivial extension with respect to~\eqref{diag:first Birkhoff subcategory} if and only if
\begin{enumerate}
	\item the morphism $l$ is an isomorphism,
	\item $\left[f,l\right]\colon \left[M',L'\right]\to \left[M,L\right]$ is an isomorphism.
\end{enumerate}
By definition $(f,l)$ is trivial with respect to~\eqref{diag:first Birkhoff subcategory}, if and only if the cube on the right in Figure~\ref{Fig:Commutator}
\begin{figure}
\resizebox{.5\textwidth}{!}
{$\xymatrix@!0@=4em{
\left[M',L'\right]
\ar[rr]
\ar[rd]_-{\left[f,l\right]}
&&
M'
\ar[rr]
\ar[dd]|(0.5){\hole}_(0.33){\partial'}
\ar[rd]|-{f}
&&
\frac{M'}{\left[M',L'\right]}
\ar[dd]|-{\hole}
\ar[rd]
\\
&
\left[M,L\right]
\ar[rr]
&&
M
\ar[rr]
\ar[dd]_(0.33){\partial}
&&
\frac{M}{\left[M,L\right]}
\ar[dd]
\\
&&
L'
\ar[rr]|(0.5){\hole}
\ar[rd]_-{l}
&&
0
\ar@{=}[rd]
\\
&&&
L
\ar[rr]
&&
0
}$}
\caption{Induced morphism of commutators.}\label{Fig:Commutator}	
\end{figure}
is a pullback in $\XMod(\A)$. Since pullbacks are computed levelwise in~$\XMod(\A)$, this is the same as asking that both the top and the bottom faces are pullbacks in $\A$. Now the top face is a pullback if and only if $\left[f,l\right]$ is an isomorphism; the bottom face is a pullback if and only if $l$ is an isomorphism as well.

The next step is showing that for any extension~\eqref{eq:extension of crossed modules} which is central with respect to~\eqref{diag:first Birkhoff subcategory}, $l$ is an isomorphism. In order to do so, recall that $(f,l)$ is central when there exists an extension 
\[
(\widetilde{M}\xrightarrow{\widetilde{\partial}}\widetilde{L},\widetilde{\xi})\xrightarrow{(g,k)}(M\xrightarrow{\partial}L,\xi)
\]
such that the pullback $(\overline{f},\overline{l})$ of $(f,l)$ along $(g,k)$ is trivial. By looking at the pullback
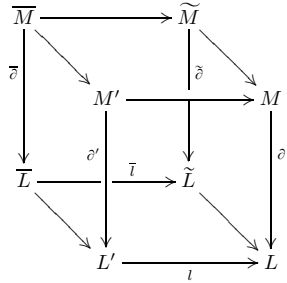
\[
\resizebox{.3\textwidth}{!}
{$\xymatrix@!0@=4em{
\overline{M}
\ar[rr]
\ar[dd]_(0.33){\overline{\partial}}
\ar[rd]
&&
\widetilde{M}
\ar[dd]|-{\hole}^(0.33){\widetilde{\partial}}
\ar[rd]
\\
&
M'
\ar[rr]
\ar[dd]_(0.33){\partial'}
&&
M
\ar[dd]^(0.33){\partial}
\\
\overline{L}
\ar[rr]^(0.66){\overline{l}}|(0.5){\hole}
\ar[rd]
&&
\widetilde{L}
\ar[rd]
\\
&
L'
\ar[rr]_-{l}
&&
L
}$}
\]
and by using the condition for trivial extensions obtained above, we know that $\overline{l}$ is an isomorphism. Hence $l$ is an isomorphism as well: on the one hand, $l$ is a regular epimorphism by hypothesis, being part of an extension; on the other hand, it is a monomorphism, since $\overline{l}$ is so and because pullbacks reflect monomorphisms.
\end{proof}

\begin{proof}[Proof of Proposition~\ref{prop:equivalence between two concepts of universal central extension}]
We already know that in order for an extension to be central with respect to~\eqref{diag:first Birkhoff subcategory}, it has to have an isomorphism in the second component. Let us therefore fix an extension $(f,1_L)$. Consider its kernel pair in Figure~\ref{FigKernelPair} 
\begin{figure}
\resizebox{.3\textwidth}{!}
{$
\xymatrix@!0@=4em{
M'\times_MM'
\ar[rr]^-{\pi_1}
\ar[dd]_-{\partial''}
\ar[rd]_-{\pi_2}
&&
M'
\ar[dd]|-{\hole}^(0.33){\partial'}
\ar[rd]|-{f}
\\
&
M'
\ar[rr]^(0.33){f}
\ar[dd]_(0.33){\partial'}
&&
M
\ar[dd]^-{\partial}
\\
L
\ar@{=}[rr]|(0.5){\hole}
\ar@{=}[rd]
&&
L
\ar@{=}[rd]
\\
&
L
\ar@{=}[rr]
&&
L
}$}
\caption{The kernel pair of $(f,1_L)$.}\label{FigKernelPair}
\end{figure}
and in particular one of the two projections $(\pi_1,1_L)$. We will use the following chain of equivalent conditions to obtain the claim:
\begin{tfae}
\item $(f,1_L)$ is central with respect to~\eqref{diag:first Birkhoff subcategory},
\item $(\pi_1,1_L)$ is trivial with respect to~\eqref{diag:first Birkhoff subcategory},
\item $K_{[\pi_1,1_L]}=0$,
\item $[K_{\pi_1},L]=0$,
\item $[K_f,L]=0$,
\item $(f,1_L)$ is central with respect to~\eqref{diag:second Birkhoff subcategory}.
\end{tfae}
The equivalence between (i) and (ii) is given by the fact that the extension $(f,1_L)$ is central with respect to~\eqref{diag:first Birkhoff subcategory} if and only if it is normal with respect to~\eqref{diag:first Birkhoff subcategory}.

To show (ii) if and only if (iii) we use Lemma~\ref{lemma:equivalent condition to triviality} and the fact that $[\pi_1,1_L]$ is already a split epimorphism by construction (it is defined through a kernel pair): this means that it is an isomorphism if and only if its kernel $K_{[\pi_1,1_L]}$ is trivial. Now consider the diagram 
\begin{equation}
\label{diag:application of protoadditivity}
\vcenter{
\xymatrixcolsep{3pc}
\xymatrix{
\left[K_{\pi_1},L\right] \ar[d] \ar[r]^-{k_{[\pi_1,1_L]}} & \left[M'\times_MM',L\right] \ar[d] \ar@<.5ex>[r]^-{[\pi_1,1_L]} & \left[M',L\right] \ar[d] \ar@<.5ex>[l]^-{[\Delta_{M'},1_L]}\\
K_{\pi_1} \ar[r]_-{k_{\pi_1}} & M'\times_MM' \ar@<.5ex>[r]^-{\pi_1} & M' \ar@<.5ex>[l]^-{\Delta_{M'}}
}
}
\end{equation}
The functor that sends an $L$-crossed module $(\partial\colon{M\to L},\xi)$ to $\left[M,L\right]$ is protoadditive (see Theorem~\ref{thm:F is protoadditive}) and hence the first row in~\eqref{diag:application of protoadditivity} is again a split short exact sequence: this means that $K_{\left[\pi_1,1\right]}\cong \left[K_{\pi_1},L\right]$, that is (iii) if and only if (iv). The equivalence between (iv) and (v) is simply given by the vertical isomorphism on the left of the diagram
\[
\xymatrixcolsep{3pc}
\xymatrix{
K_{\pi_1} \ar@{.>}[d] \ar[r]^-{k_{\pi_1}} & M'\times_MM' \pullback \ar[d]_-{\pi_2} \ar[r]^-{\pi_1} & M' \ar[d]^-{f}\\
K_{f} \ar[r]_-{k_f} & M' \ar[r]_-{f} & M
}
\]
due to the fact that the square on the right is a pullback by construction. The last step is given by Corollary~\ref{cor:equivalent condition to centrality}.

Finally, it is easy to see that a central extension is universal with respect to~\eqref{diag:first Birkhoff subcategory} if and only if it is universal with respect to~\eqref{diag:second Birkhoff subcategory}.
\end{proof}

\begin{proposition}
\label{prop:equivalence between two concepts of perfect object}
An $L$-crossed module $(\partial\colon{M\to L},\xi)$ is perfect in the sense of Definition~\ref{defi:perfect object in XMod_L(A) external} if and only if it is perfect with respect to the Birkhoff subcategory $\Ab(\A)$ when seen as an object in $\XMod(\A)$.
\end{proposition}
\begin{proof}
The crossed module $(\partial\colon{M\to L},\xi)$ is perfect with respect to Definition~\ref{defi:perfect object in XMod_L(A) external} if and only if $\left[L,M\right]=M$. This amounts to ${M}/{\left[L,M\right]}=0$, which in turn is the same as $F(\partial\colon{M\to L},\xi)=0$, that is perfectness with respect to Definition~\ref{defi:perfect object in pointed setting}.
\end{proof}

\section*{Acknowledgements}
We would like to thank Alan S.~Cigoli, Marino Gran, Sandra Mantovani, Giuseppe Metere and Andrea Montoli for helpful comments and suggestions. The first author thanks the Université catholique de Louvain for its kind hospitality during his stays in Louvain-la-Neuve. The second author is grateful to the Università degli Studi di Milano and the Università degli Studi di Palermo for their kind hospitality during his stays in Milan and in Palermo.

\renewcommand{\MR}[1]{}


\begin{thebibliography}{10}

\bibitem{BB04}
F.~Borceux and D.~Bourn, \emph{Mal'cev, protomodular, homological and
  semi-abelian categories}, Mathematics and its Applications, vol. 566, Kluwer
  Academic Publishers, Dordrecht, 2004. \MR{2044291}

\bibitem{BJ01}
F.~Borceux and G.~Janelidze, \emph{Galois theories}, Cambridge Stud. Adv.
  Math., vol.~72, Cambridge Univ. Press, 2001.

\bibitem{Bou01}
D.~Bourn, \emph{{$3\times3$} lemma and protomodularity}, J. Algebra
  \textbf{236} (2001), no.~2, 778--795. \MR{1813501}

\bibitem{BG02}
D.~Bourn and M.~Gran, \emph{Central extensions in semi-abelian categories}, J.
  Pure Appl. Algebra \textbf{175} (2002), no.~1-3, 31--44.

\bibitem{BJ98}
D.~Bourn and G.~Janelidze, \emph{Protomodularity, descent, and semidirect
  products}, Theory Appl. Categ. \textbf{4} (1998), no.~2, 37--46.

\bibitem{Br82}
K.~S. Brown, \emph{Cohomology of groups}, Grad. Texts in Math., vol.~87,
  Springer, 1982.

\bibitem{BL87}
R.~Brown and J.-L. Loday, \emph{Van {K}ampen theorems for diagrams of spaces},
  Topology \textbf{26} (1987), no.~3, 311--335.

\bibitem{CJ03}
A.~Carboni and G.~Janelidze, \emph{Smash product of pointed objects in
  lextensive categories}, J.~Pure Appl.~Algebra \textbf{183} (2003), 27--43.

\bibitem{CKP93}
A.~Carboni, G.~M. Kelly, and M.~C. Pedicchio, \emph{Some remarks on {M}al'tsev
  and {G}oursat categories}, Appl. Categ. Structures \textbf{1} (1993), no.~4,
  385--421. \MR{1268510}

\bibitem{CVdL14}
J.~M. Casas and T.~Van~der Linden, \emph{Universal central extensions in
  semi-abelian categories}, Appl. Categ. Structures \textbf{22} (2014), no.~1,
  253--268. \MR{3163517}

\bibitem{CGVdL15b}
A.~S. Cigoli, J.~R.~A. Gray, and T.~Van~der Linden, \emph{Algebraically
  coherent categories}, Theory Appl. Categ. \textbf{30} (2015), no.~54,
  1864--1905.

\bibitem{dMVdL19.3}
D.~di~Micco and T.~Van~der Linden, \emph{An intrinsic approach to the
  non-abelian tensor product via internal crossed squares}, preprint \texttt{arXiv:1911.08781}, 2019.

\bibitem{Eda19}
B.~Edalatzadeh, \emph{Universal central extensions of {L}ie crossed modules
  over a fixed {L}ie algebra}, Appl.\ Categ.\ Structures \textbf{27} (2019),
  no.~2, 111--123.

\bibitem{EG-honfg}
T.~Everaert and M.~Gran, \emph{Homology of $n$-fold groupoids}, Theory Appl.
  Categ. \textbf{23} (2010), no.~2, 22--41.

\bibitem{Everaert-Gran-TT}
T.~Everaert and M.~Gran, \emph{Protoadditive functors, derived torsion theories
  and homology}, J.~Pure Appl.\ Algebra \textbf{219} (2015), no.~8, 3629--3676.

\bibitem{EverVdL1}
T.~Everaert and T.~Van~der Linden, \emph{{B}aer invariants in semi-abelian
  categories~{I}: {G}eneral theory}, Theory Appl. Categ. \textbf{12} (2004),
  no.~1, 1--33.

\bibitem{EverVdL2}
T.~Everaert and T.~Van~der Linden, \emph{{B}aer invariants in semi-abelian categories~{II}: {H}omology},
  Theory Appl. Categ. \textbf{12} (2004), no.~4, 195--224.

\bibitem{EverVdLRCT}
T.~Everaert and T.~Van~der Linden, \emph{Relative commutator theory in semi-abelian categories}, J.~Pure
  Appl. Algebra \textbf{216} (2012), no.~8--9, 1791--1806.

\bibitem{Gra99}
M.~Gran, \emph{Internal categories in {M}al' cev categories}, J. Pure Appl.
  Algebra \textbf{143} (1999), no.~1-3, 221--229.
  
\bibitem{Gra02}
M.~Gran, \emph{Commutators and central extensions in universal algebra}, J.
  Pure Appl. Algebra \textbf{174} (2002), no.~3, 249--261. \MR{1929407}

\bibitem{GVdL08}
M.~Gran and T.~Van~der Linden, \emph{On the second cohomology group in
  semi-abelian categories}, J.~Pure Appl.\ Algebra \textbf{212} (2008),
  636--651.

\bibitem{HL13}
M.~Hartl and B.~Loiseau, \emph{On actions and strict actions in homological
  categories}, Theory Appl.~Categ. \textbf{27} (2013), no.~15, 347--392.

\bibitem{HVdL11}
M.~Hartl and T.~Van~der Linden, \emph{The ternary commutator obstruction for
  internal crossed modules}, Adv. Math. \textbf{232} (2013), 571--607.
  \MR{2989994}

\bibitem{Hig56}
P.~J. Higgins, \emph{Groups with multiple operators}, Proc. Lond. Math. Soc.
  (3) \textbf{6} (1956), no.~3, 366--416.

\bibitem{Jan90}
G.~Janelidze, \emph{Pure {G}alois theory in categories}, J. Algebra
  \textbf{132} (1990), no.~2, 270--286. \MR{1061480}

\bibitem{Jan03}
G.~Janelidze, \emph{Internal crossed modules}, Georgian Math. J. \textbf{10} (2003),
  no.~1, 99--114. \MR{1990690}

\bibitem{JK94}
G.~Janelidze and G.~M. Kelly, \emph{Galois theory and a general notion of
  central extension}, J. Pure Appl. Algebra \textbf{97} (1994), no.~2,
  135--161. \MR{1312759}

\bibitem{JK00}
G.~Janelidze and G.~M. Kelly, \emph{Central extensions in universal algebra: a unification of three
  notions}, Algebra Universalis \textbf{44} (2000), no.~1-2, 123--128.
  \MR{1801637}

\bibitem{JMT02}
G.~Janelidze, L.~M{\'a}rki, and W.~Tholen, \emph{Semi-abelian categories},
  J.~Pure Appl. Algebra \textbf{168} (2002), no.~2--3, 367--386.

\bibitem{MM10}
S.~Mantovani and G.~Metere, \emph{Normalities and commutators}, J. Algebra
  \textbf{324} (2010), no.~9, 2568--2588. \MR{2684155}

\bibitem{MFVdL12}
N.~Martins-Ferreira and T.~Van~der Linden, \emph{A note on the ``{S}mith is
  {H}uq'' condition}, Appl. Categ. Structures \textbf{20} (2012), no.~2,
  175--187. \MR{2899723}

\bibitem{Orz72}
G.~Orzech, \emph{Obstruction theory in algebraic categories {I} and {II}},
  J.~Pure Appl. Algebra \textbf{2} (1972), 287--314 and 315--340.

\bibitem{Ped95b}
M.~C. Pedicchio, \emph{A categorical approach to commutator theory}, J. Algebra
  \textbf{177} (1995), no.~3, 647--657. \MR{1358478}

\end{thebibliography}

\end{document}